\documentclass{article}
\usepackage{hyperref}
\usepackage{amssymb,amsmath,amscd,dsfont}
\usepackage{amsthm}

\newcommand{\ol}[1]{\ensuremath{\overline{#1}}}

\newcommand{\cD}{\ensuremath{\mathcal{D}}}

\newcommand{\cP}{\ensuremath{\mathcal{P}}}

\newcommand{\cH}{\ensuremath{\mathcal{H}}}

\newcommand{\cL}{\ensuremath{\mathcal{L}}}
\newcommand{\cU}{\ensuremath{\mathcal{U}}}

\newcommand{\fg}{\ensuremath{\mathfrak{g}}}

\newcommand{\fk}{\ensuremath{\mathfrak{k}}}
\newcommand{\fr}{\ensuremath{\mathfrak{r}}}

\newcommand{\Ad}{\ensuremath{\operatorname{Ad}}}
\newcommand{\U}{\ensuremath{\operatorname{U}}}

\newcommand{\Aut}{\ensuremath{\operatorname{Aut}}}

\newcommand{\shalf}{\ensuremath{{\textstyle \frac{1}{2}}}}
\newcommand{\R}{\ensuremath{\mathbb{R}}}
\newcommand{\C}{\ensuremath{\mathbb{C}}}

\newcommand{\Z}{\ensuremath{\mathbb{Z}}}
\newcommand{\N}{\ensuremath{\mathbb{N}}}
\newcommand{\Q}{\ensuremath{\mathbb{Q}}}
\newcommand{\T}{\ensuremath{\mathbb{T}}}
\newcommand{\bv}{{\bf{v}}}
\newcommand{\bxi}{{\boldsymbol\xi}} 
\newcommand{\bP}{{\mathbb P}}
\newcommand{\bS}{{\mathbb S}}

\newcommand{\one}{\ensuremath{\mathbf{1}}}

\newcommand{\Exp}{\ensuremath{\operatorname{Exp}}}
\newcommand{\Diff}{\ensuremath{\operatorname{Diff}}}

\renewcommand{\tilde}{\widetilde}
\newcommand{\into}{\hookrightarrow}
\newcommand{\Spann}{\mathop{{\rm span}}\nolimits}

\newcommand{\subeq}{\subseteq}
\newcommand{\g}{{\mathfrak g}}
\newcommand{\PU}{\mathop{\rm PU{}}\nolimits}

\newcommand{\Ext}{\mathop{\rm Ext}\nolimits}

\newcommand{\End}{\mathop{\rm End}\nolimits}
\newcommand{\la}{\langle}
\newcommand{\ra}{\rangle}

\newcommand{\arr}{\hbox to 60pt{\rightarrowfill}}
\newcommand{\sssarr}{\hbox to 20pt{\rightarrowfill}}
\newcommand{\ssarr}{\hbox to 30pt{\rightarrowfill}}
\newcommand{\sarr}{\hbox to 40pt{\rightarrowfill}}

\newcommand{\sssmapright}[1]{\smash{\mathop{\sssarr}\limits^{#1}}}

\newcommand{\smapright}[1]{\smash{\mathop{\sarr}\limits^{#1}}}

\theoremstyle{plain}
\newtheorem{Theorem}{Theorem}[section]
\newtheorem{Lemma}[Theorem]{Lemma}
\newtheorem{Proposition}[Theorem]{Proposition}
\newtheorem{Corollary}[Theorem]{Corollary}

\theoremstyle{definition}
\newtheorem{Definition}[Theorem]{Definition}

\newtheorem{Remark}[Theorem]{Remark}
\newtheorem{Example}[Theorem]{Example}

\renewcommand{\:}{\colon}
\newcommand{\1}{\mathbf{1}}
\newcommand{\res}{\vert}

\renewcommand{\hat}{\widehat} 
\renewcommand{\phi}{\varphi}
\newcommand{\dd}{{\tt d}}

\newcommand\oline{\overline}

\begin{document}

\title{Projective unitary representations of infinite dimensional Lie groups} 
\author{Bas Janssens and  Karl-Hermann Neeb}
\maketitle

\begin{abstract}
For an infinite dimensional Lie group $G$ modelled on a locally convex Lie algebra $\fg$, 
we prove that every smooth \emph{projective} unitary representation of $G$ corresponds to a smooth
\emph{linear} unitary representation of a Lie group extension $G^{\sharp}$ of $G$.
(The main point is the smooth structure {on~$G^{\sharp}$}.)
For infinite dimensional Lie groups $G$ which are 1-connected, regular, and modelled on a barrelled 
Lie algebra $\fg$, we characterize the unitary $\fg$-representations which integrate to $G$.
Combining these results, we give a precise formulation of the correspondence
between smooth projective unitary representations of $G$, smooth linear
unitary representations of~$G^{\sharp}$, and the appropriate 
unitary representations of its Lie algebra $\fg^{\sharp}$.   
\end{abstract}


\tableofcontents

\section*{Introduction} 

The state space of a quantum mechanical system is described by the 
projective space $\bP(\cH)$ of a Hilbert space $\cH$.
If a 
topological group
$G$ is a 
symmetry group of this 
system, 
then, if $G$ is connected, we have a homomorphism $\oline\rho \: G \to \PU(\cH)$ into the 
projective unitary group of $\cH$, i.e., a projective unitary representation 
of~$G$. 
The natural continuity requirement for a projective representation 
is that all its orbit maps are continuous.
This is equivalent to continuity of the action of $G$ on $\bP(\cH)$,
and to continuity of $\ol\rho$ 
with 
respect to the quotient topology on $\PU(\cH) \cong \U(\cH)_s/\T \1$, where 
$\U(\cH)_s$ denotes the unitary group endowed with the strong 
operator topology (\cite{Ri79}, \cite{Ne14}).  Since $\U(\cH)_s$ is a central 
topological 
$\T$-extension of $\PU(\cH)$, we can always pull this extension  back to $G$  
to obtain a central $\T$-extension $G^\sharp$ of $G$ in the topological category, 
and a continuous unitary representation $\pi \: G^\sharp \to G$ lifting 
the projective representation $\oline\pi$. 
In this context, the classification 
of projective unitary representations of $G$ breaks up into two parts:
determining the group $\Ext_{\rm top}(G,\T)$ of all topological central 
$\T$-extensions $G^\sharp$ of $G$, 
and classifying the unitary representations 
of the individual central extensions $G^\sharp$. 
This strategy works particularly well if 
$G$ is a finite dimensional Lie group, because then $G^\sharp$ always carries a 
natural Lie group structure (\cite{Ba54, Va85}). On the Lie algebra 
level, it now suffices to study (infinitesimally) unitary representations 
on pre-Hilbert spaces and to develop criteria for their integrability 
to groups; see \cite{Nel69} for some of the most powerful criteria in 
this context. 

In this paper, we want to address these issues 
in a systematic way for infinite dimensional 
Lie groups, more precisely, for Lie groups $G$ modelled on a locally convex 
topological vector space. 
Because
the topological group $\U(\cH)_s$ does not carry a 
Lie group structure,
one traditionally deals with 
these groups simply as topological groups, using
the exponential function $\exp\: \g \to G$ only as a means to parametrise
one-parameter subgroups of $G$ (cf.\ \cite{Wa98}, \cite{Lo94}, \cite{Se81}). 
By Stone's Theorem, one-parameter subgroups of $\U(\cH)_s$ correspond to skew-adjoint, 
possibly unbounded operators on $\cH$, so that
every continuous unitary representation 
$\rho \: G \to \U(\cH)$ gives rise to a family 
$\dd\rho(\xi)$, $\xi \in \g$, of skew-adjoint unbounded operators 
with different domains $\cD_\xi$.

This context is not quite suitable for Lie theory though, because due to possible domain issues,
one may not get a representation at the infinitesimal
level.  
%
In order to obtain a Lie algebra representation,
it is natural to require that the representation
$(\rho,\cH)$ be {\it smooth}, in the sense that the space
$\cH^\infty$ of vectors with a smooth orbit map is dense in $\cH$.
This ensures that $\dd\rho$ becomes
a 
Lie algebra representation on the dense common domain $\cH^{\infty}$, and with that, the full power of 
Lie Theory, including infinitesimal methods, becomes available.
For finite dimensional Lie groups, smoothness 
follows from continuity by G\aa{}rding's Theorem \cite{Ga47},
but this is no longer true for infinite 
dimensional Lie groups (\cite{BN08, Ne10b}). 
For those,
smoothness has to be imposed as an extra technical condition, 
justified by virtually all important examples.
In the same vein, we define a projective unitary representation 
$\oline\rho \: G\to \PU(\cH)$ to be {\it smooth} 
if the set of rays with a smooth orbit map is dense in $\bP(\cH)$. 


The central results of this paper are the following: 
\begin{itemize}
\item If $G$ is a Lie group (modelled on an arbitrary
locally convex vector space), then
every smooth projective unitary representation of  $G$ 
corresponds to a smooth unitary representation of a central Lie group extension 
$G^\sharp$ of $G$ (Theorem~\ref{thm:1.4}). 
The nontrivial point here is that $G^{\sharp}$ is not just a topological group,
but a locally convex Lie group. 
\item For Lie groups $G$ that are $1$-connected, regular, and modelled on a barrelled 
locally convex space,
we characterise which
unitary representations of its Lie algebra $\fg$ are integrable to a group representation 
in terms of 
existence and smoothness of solutions to certain linear 
initial value problems (Theorem~\ref{heenenweer}).
%
\item We clarify the correspondence between 
smooth projective representations, unitary representations of the 
corresponding central extensions, and the corresponding data on the Lie algebra 
level in terms of suitable categories (Theorem~\ref{maintheorem}), 
with due care 
for the appropriate intertwiners.
\end{itemize}

We now describe our results and the structure of the paper in more 
detail. In Section~\ref{SectionPPER}, 
we introduce Lie groups modelled on locally convex spaces and in 
particular the notion of a regular Lie group. 
In Section~\ref{Sec2}, we 
introduce the appropriate notions of \emph{smoothness} and \emph{continuity}
for unitary and projective unitary representations of locally convex Lie groups. 

The original part of this paper starts in Section~\ref{sec:3}, where 
we discuss the passage back and forth between 
smooth unitary representations of a locally convex Lie group $G$ with 
Lie algebra $\g$  (the global level),
and the corresponding
unitary representations 
of $\g$ on a pre-Hilbert space (the infinitesimal 
level). 
A key observation is that a smooth unitary representation 
$(\rho, \cH)$ of a connected Lie group $G$ is always determined 
uniquely by  the derived representation of $\g$ 
on the space of smooth vectors. We give a short and direct proof in 
Proposition~\ref{LArulez}. 
A much harder problem is to 
characterise 
those unitary Lie algebra representations $(\pi,V)$ which 
can be integrated to a smooth group representation.
The main idea here is to equip $V$ with two locally convex topologies derived from $\pi$,
the \emph{weak} and the \emph{strong} topology, which are finer than
the \emph{norm topology} of the pre-Hilbert space $V$.
The advantage of these topologies on $V$ is that unlike the 
norm topology,
they make 
the infinitesimal action $\fg \times V \rightarrow V$, defined by 
$(\xi, \psi) \mapsto \pi(\xi)\psi$, into a
sequentially continuous map (Lemma~\ref{seqcont}). 
Although this map is continuous for Banach--Lie groups, 
this fails in general for Fr\'echet--Lie groups.
If $(\pi,V)$ is the derived $\fg$-representation of the smooth unitary $G$-representation $(\rho, \cH)$, 
then it is  {\it regular} 
in the sense that,  for every smooth curve $\xi \: [0,1] \to \g$ and $\psi \in V$, 
the initial value problem 
\[
{\textstyle \frac{d}{dt}} \psi_{t} = \pi(\xi_{t}) \psi_{t}, \quad \psi_0 = \psi
\]
has a smooth solution (in the weak or strong topology on $V$, which is a stronger
requirement than having a solution in the norm topology!),
 and $\psi_1$ depends smoothly on $\xi\in C^\infty([0,1],\g)$ with respect to 
the topology of uniform convergence of all derivatives. 
The main result of Section~\ref{sec:3} is Theorem~\ref{heenenweer}, 
which states that for regular, 1-connected Lie groups $G$ with a barrelled Lie algebra $\fg$, 
a unitary $\fg$-representation is integrable to a $G$-representation 
if and only if it is regular.
Along the way, we prove that if $(\rho, \cH)$ is a smooth unitary representation of a 
Fr\'echet--Lie group $G$, then its space $\cH^\infty$ of smooth vectors is complete with respect to both 
the weak and the strong topology (Proposition~\ref{completeness}).  
Sufficient criteria for regularity have been developed for various types of Lie algebras, 
such as loop algebras and the Virasoro algebra, in \cite{TL99}.

In Section~\ref{sec:4}, we introduce the concept of a smooth 
projective unitary representation $\oline\rho \: G \to \PU(\cH)$ of a 
locally convex Lie group $G$. We show that any such representation 
gives rise to a smooth unitary representation of a 
central Lie group extension $G^{\sharp}$ of $G$ by the circle group $\T$, and hence 
to a representation of the central Lie algebra 
extension $\fg^{\sharp}$ of $\fg$. This general result 
subsumes various weaker results and ad hoc considerations 
(\cite{Ne10a}, \cite{PS86}, \cite{Se81}, \cite{Mc85, Mc89}) 
and applies 
in particular to many important classes  of groups, such as 
loop groups, the Virasoro group, and restricted unitary groups.  
This breaks the problem of classifying smooth projective unitary representations 
into two smaller parts: 
determining the relevant central Lie group extensions $G^{\sharp}$ of $G$ by $\T$
(cf.~\cite{Ne02}), 
and classifying  the unitary representations 
of the individual central extensions $G^{\sharp}$.
In \cite{JN15}, we use this strategy for the classification 
of projective unitary positive energy representations of gauge groups.

In Section~\ref{AppendixA}, we
refine the techniques used in 
\cite{Ne10b, Ne14} 
to provide 
effective criteria for the smoothness of a ray $[\psi] \in \bP(\cH)$, and to 
determine the structure of the subset $\bP(\cH)^{\infty}$ of all smooth rays as a 
union of subsets $\cP(\cD_j)$, where $(\cD_j)_{j \in J}$ are mutually orthogonal 
subspaces of $\cH$. In particular, this shows that for smooth projective representations, 
the dense set of smooth rays is a projective space.

The Lie algebra extension $\g^\sharp \rightarrow \fg$ of the central $\T$-extension 
$G^{\sharp} \rightarrow G$ 
is a central extension of $\g$ by $\R$, hence 
determined by a cocycle $\omega \: \g \times \g \to \R$. 
In Section~\ref{sec:6}, we describe the Lie algebra cocycles
in terms of smooth 
rays, 
and formulate necessary conditions for a Lie algebra 
cocycle to come from a projective unitary representation. 
One is  the existence of an equivalent cocycle 
which is the imaginary part of a positive semi-definite hermitian form on 
$\g_\C$ (Proposition~\ref{polarisability}). 

In Section~\ref{sec:7}, we 
formulate the correspondence between 
smooth projective unitary representations of $G$, smooth unitary representations of 
$G^{\sharp}$, and the corresponding data on the Lie algebra 
side in terms of suitable categories, i.e., taking intertwiners into account.  

In Section~\ref{sec:8}, we briefly discuss the special case where 
the Lie group under consideration is a semidirect product 
$G \rtimes_\alpha R$. 
This setting, especially with $R = \R$, is frequently encountered 
in the representation theory of infinite dimensional Lie groups, 
in particular for loop groups and gauge groups 
(cf.\ \cite{PS86, JN15}).
For $R = \T$ or $R = \R$, we thus encounter Lie algebras of the form 
$\g \rtimes_D \R$, where $D \: \g \to \g$ is a continuous derivation. 
Such a derivation is called admissible if its kernel and cokernel 
split in a natural way, a requirement automatically fulfilled for $R = \T$. 
In Section~\ref{AppendixB} we show how the admissibility of 
$D$ can be used to compute $H^2(\g\rtimes_D \R,\R)$ in terms of 
$D$-invariant cocycles on $\g$. We further show that the Stone--von Neumann 
Theorem implies that cocycles on $\g \rtimes_D \R$ arising from 
projective positive energy representations lead to $D$-invariant cocycles on~$\g$. 

Section~\ref{sec:10} concludes this paper with a discussion of projective unitary 
representations for three classes of 
locally convex Lie groups:
abelian groups, whose central extensions are related to Heisenberg groups;
the group $\Diff(\bS^1)_+$ of orientation preserving diffeomorphisms of the circle,
whose central extensions are related to the Virasoro algebra;
and twisted loop groups, whose central extensions are related to affine Kac--Moody 
algebras.

\subsection*{Notation and terminology} 

We denote the unit circle $\{ z \in \C \: |z| =1\}$ by $\bS^1$ if we consider 
it as a manifold, and by $\T$ if we consider it as a Lie group.  
The exponential function $\exp \colon \R \rightarrow \T$ is given by
$\exp(t) = e^{2\pi i t}$ with kernel $\Z$.
For a complex Hilbert space $\cH$, we take
the scalar product $\langle \,\cdot\,,\,\cdot\, \rangle$ to be linear in the \emph{second} argument.
The projective space is denoted $\bP(\cH)$, and we write
$[\psi] := \C \psi$ for the ray generated by a nonzero vector $\psi \in \cH$. 
We denote the group of unitary operators of $\cH$ by
$\U(\cH)$, and write $\PU(\cH) := \U(\cH)/\T \1$ for the projective unitary group.
The image of $U \in \U(\cH)$ in $\PU(\cH)$ is denoted $\oline U$.
Locally convex vector spaces are always assumed to be Hausdorff.
 
\section{Representations of locally convex Lie groups}\label{SectionPPER}

In this section, we introduce Lie groups modelled on locally convex spaces, 
or \emph{locally convex Lie groups} for short. This is a
generalisation of the concept of a finite dimensional Lie group
that captures a wide range of interesting examples 
(cf.\ \cite{Ne06} for an overview). 

\subsection{Smooth functions}

Let $E$ and $F$ be locally convex spaces, $U
\subeq E$ open and $f \: U \to F$ a map. Then the {\it derivative
  of $f$ at $x$ in the direction $h$} is defined as 
$$ \partial_{h}f(x) := \lim_{t \to 0} \frac{1}{t}(f(x+th) -f(x)) $$
whenever it exists. We set $D f(x)(h) :=  \partial_{h}f(x)$. The function $f$ is called {\it differentiable at
  $x$} if $D f(x)(h)$ exists for all $h \in E$. It is called {\it
  continuously differentiable} if it is differentiable at all
points of $U$ and 
$$ D f \: U \times E \to F, \quad (x,h) \mapsto D f(x)(h) $$
is a continuous map. Note that this implies that the maps 
$D f(x)$ are linear (cf.\ \cite[Lemma~2.2.14]{GN15}). 
The map $f$ is called a {\it $C^k$-map}, $k \in \N \cup \{\infty\}$, 
if it is continuous, the iterated directional derivatives 
$$ D^{j}f(x)(h_1,\ldots, h_j)
:= (\partial_{h_j} \cdots \partial_{h_1}f)(x) $$
exist for all integers $1\leq j \leq k$, $x \in U$ and $h_1,\ldots, h_j \in E$, 
and all maps $D^j f \: U \times E^j \to F$ are continuous. 
As usual, $C^\infty$-maps are called {\it smooth}. 

\subsection{Locally convex Lie groups}

Once the concept of a smooth function 
between open subsets of locally convex spaces is established, it is clear how to define 
a locally convex smooth manifold (cf.\ \cite{Ne06}, \cite{GN15}). 
\begin{Definition}(Locally convex Lie groups)
A {\it locally convex Lie group} $G$ is a group equipped with a 
smooth manifold structure modelled on a locally convex space 
for which the group multiplication and the 
inversion are smooth maps. 
A {\it morphism of Lie groups} is a smooth group homomorphism.
\end{Definition}
We write $\1 \in G$ for the identity element. 
Its Lie algebra $\g$ is identified with 
the tangent space $T_\1(G)$, and the Lie bracket is obtained by identification with the 
Lie algebra of left invariant vector fields. It is a \emph{locally convex Lie algebra}
in the following sense.
\begin{Definition}(Locally convex Lie algebras)
A {\it locally convex Lie algebra} is a 
locally convex space $\fg$ with a continuous Lie bracket
$[\,\cdot\,,\,\cdot\,] \colon \fg \times \fg \rightarrow \fg$. 
Homomorphisms of locally convex Lie algebras are continuous Lie algebra 
homomorphisms.
\end{Definition}

The right action $R_{g} \colon h \mapsto hg$ of $G$ on itself 
induces a right action on $T(G)$,
which we denote $(v,g) \mapsto v \cdot g$.

\begin{Definition}(Logarithmic derivatives)
For a smooth map $\gamma \colon M \rightarrow G$,
the \emph{right logarithmic derivative} $\delta^{R}\gamma \colon TM \rightarrow \fg$ is
defined by 
\[\delta^{R}\gamma := T\gamma \cdot \gamma^{-1}, \quad i.e., \quad 
(\delta^R\gamma)(v_m) = T_m(\gamma)v_m \cdot \gamma(m)^{-1}.\] 
\end{Definition}
If $x$ is a coordinate on $M$, we will often write
$\delta^{R}_{x}\gamma$ instead of $\delta^{R}_{\partial_{x}}\gamma$.
The logarithmic derivative satisfies
the \emph{Maurer--Cartan equation} 
\begin{eqnarray}\label{MChammer}
T_{v}\delta^{R}_{w}\gamma - T_{w}\delta^{R}_{v}\gamma = [\delta^{R}_{v}\gamma , \delta^{R}_{w}\gamma]\,.
\end{eqnarray}

The notion of a \emph{regular} Lie group, introduced by Milnor~\cite{Mi84}, is 
especially useful if one wishes to pass from the infinitesimal to the 
global level. Regularity is satisfied by all major classes of locally convex Lie groups \cite{Ne06}. 
\begin{Definition}(Regularity) A locally convex Lie group $G$ is called \emph{regular} 
if for every smooth map $\xi \colon [0,1] \rightarrow \fg$, the differential equation
\[\delta^{R}_{t}  \gamma  =  \xi(t)\]
with initial condition $\gamma(0) = \one$ has a solution 
$\gamma \colon [0,1] \rightarrow G$ and $\gamma(1)$ depends smoothly on $\xi$.
If such a solution exists, then it is automatically unique~\cite[II.3.6(c)]{Ne06}.
\end{Definition}

Since the emphasis in the present paper is more on the global level and 
on the passage from the global to the infinitesimal level, 
the majority of our results will not require regularity.

%

\section{Projective unitary representations}\label{Sec2}

We introduce the appropriate notions of \emph{smoothness} and \emph{continuity}
for unitary and projective unitary representations of locally convex Lie groups. 

\subsection{Unitary representations}
Let $G$ be a locally convex Lie group with Lie algebra $\fg$.
A \emph{unitary representation} $(\rho,\cH)$ of $G$ 
is a group homomorphism $\rho\colon G \rightarrow \U(\cH)$.

\begin{Definition}(Continuous and smooth vectors)
A vector $\psi \in \cH$ is called \emph{continuous} if its orbit map 
$G \to \cH, 
g \mapsto \rho(g)\psi$ is continuous, and \emph{smooth} if its orbit map is smooth.
We denote the subspaces of continuous and smooth vectors by $\cH_{c}$ and 
$\cH^{\infty}$, respectively. 
\end{Definition}

\begin{Definition}(Continuous and smooth unitary representations) 
A unitary representation is called \emph{continuous} if $\cH_{c}$ is dense in $\cH$,
and \emph{smooth} if $\cH^{\infty}$ is dense in $\cH$.
\end{Definition}

\begin{Remark}(Smoothness is stronger than continuity)
Clearly, every 
smooth unitary representation is continuous.
For finite dimensional Lie groups, the converse also holds: every continuous unitary representation is
automatically  smooth (cf.~\cite{Ga47}). For infinite dimensional Lie groups, however,
this converse no longer holds; there exist unitary representations of Banach--Lie groups that are 
continuous, but not smooth (cf.~\cite{BN08}, \cite{Ne10b}). 
\end{Remark}

\begin{Remark}(Strongly continuous vs.\  norm continuous representations)
Since the linear subspace $\cH_{c} \subseteq \cH$ is closed, a unitary representation is continuous 
if and only if $\cH_{c} = \cH$. This is equivalent to the 
continuity of $\rho \colon G \rightarrow \U(\cH)$ 
with respect to the \emph{strong operator topology}, which coincides with the \emph{weak
operator topology} on $\U(\cH)$. 
The \emph{norm topology} on $\U(\cH)$ is finer than the weak or strong topology.
We call $(\rho,\cH)$ \emph{norm continuous} if $\rho \colon G \rightarrow \U(\cH)$
is continuous w.r.t.\ the norm topology. 
Norm continuity implies continuity, but many
interesting continuous representations, 
such as the regular representation of a finite dimensional Lie group $G$ on $L^2(G)$, 
are not norm continuous. 
\end{Remark}

\subsection{Projective unitary representations}

A \emph{projective unitary representation} $(\ol\rho,\cH)$ of a locally convex Lie group $G$ 
is a complex Hilbert space $\cH$ with
a group homomorphism $\ol{\rho}\colon G \rightarrow \PU(\cH)$. We call two 
projective unitary representations $(\ol\rho,\cH)$ and $(\ol\rho',\cH')$ 
\emph{unitarily equivalent} if there exists a 
unitary transformation $U \colon \cH \rightarrow \cH'$ such that $\ol U \circ \ol\rho(g) = \ol\rho(g)' \circ \ol U$
for all $g\in G$.

A projective unitary representation yields
an action of $G$ on the projective space $\bP(\cH)$.
Since $\bP(\cH)$ is a Hilbert manifold, we can define continuous and smooth projective representations.

\begin{Definition}(Continuous and smooth rays) \label{def:smoothray}
A ray $[\psi] \in \bP(\cH)$ is called \emph{continuous} if its orbit map 
$G \to \bP(\cH), g \mapsto \overline{\rho}(g)[\psi]$ is continuous, and \emph{smooth} if its orbit map is smooth.
We denote the sets of continuous and smooth rays by $\bP(\cH)_{c}$ and 
$\bP(\cH)^{\infty}$, respectively. 
\end{Definition}

\begin{Definition}(Continuous and smooth projective unitary representations)
A projective unitary representation $(\ol\rho,\cH)$ is called \emph{continuous} if $\bP(\cH)_{c}$ is dense in $\cH$,
and \emph{smooth} if $\bP(\cH)^{\infty}$ is dense in $\cH$.
\end{Definition}

\begin{Remark}(Topology of $\bP(\cH)$)
The topology of the Hilbert manifold $\bP(\cH)$ is easily seen to be the initial topology 
w.r.t.\ the functions $[\psi] \mapsto p([\phi]\,; [\psi])$, where 
 \[
 p([\phi] \,; [\psi]) := \frac{|\la\phi,\psi\ra|^2}{\|\phi\|^2\|\psi\|^2}
 \]
is the \emph{transition probability} between $[\phi], [\psi] \in \bP(\cH)$ when these 
are interpreted as states of a quantum mechanical system.
This in turn agrees with the topology induced by the Fubini--Study metric 
$d([\psi],[\phi]) = \arccos \sqrt{ p([\phi] \,; [\psi])}$ (\cite[Rem.~4.4(b)]{Ne14}). 
If $[\psi] \in \bP(\cH)$ is continuous (smooth), then, for all $[\phi] \in \bP(\cH)$, the transition 
probability $p([\phi]\,; \ol\rho(g)[\psi])$ varies continuously (smoothly) with $g\in G$.
\end{Remark}

\begin{Remark}(Alternative characterisations of continuity)
Since $\bP(\cH)_{c}$ is closed \cite[Lemma~5.8]{Ne14}, a projective unitary representation is continuous if and only if 
$\bP(\cH)_{c} = \bP(\cH)$. 
We endow $\PU(\cH) = \U(\cH)/\T\one$ with the pointwise topology derived from its action on $\bP(\cH)$.
This coincides with the quotient topology derived from the strong (or, equivalently, weak) topology on $\U(\cH)$ 
\cite[Prop.~5.1]{Ne14}, and
a projective unitary representation is continuous if and only if the map $\ol\rho \colon G \rightarrow \PU(\cH)$ is continuous
with respect to this topology.
The quotient topology on $\PU(\cH)$ derived from the \emph{norm} topology on $\U(\cH)$
is finer, and we call $(\ol\rho,\cH)$ \emph{norm continuous} 
if $\ol\rho$ is continuous w.r.t.\ this topology.
\end{Remark}

\section{Integration of Lie algebra representations}
\label{sec:3}

In this section, we turn to the new material in this paper. 
We discuss the passage back and forth between 
smooth unitary representations of a locally convex Lie group $G$ with 
Lie algebra $\g$  (the global level), and the corresponding
unitary representations $\g$ on a pre-Hilbert space (the infinitesimal 
level).

Subsection~\ref{enekantop}
is devoted to the following
key observation:
a smooth unitary representation $(\rho, \cH)$ 
of a locally convex Lie group $G$
is determined uniquely by its derived unitary Lie algebra representation $(\dd\rho, \cH^{\infty})$  
on the space of smooth vectors. 
In Subsection~\ref{anderekantop}, we tackle the
much harder question of finding necessary and sufficient criteria 
for unitary Lie algebra representations 
$(\pi,V)$ to integrate to the group level. 
This culminates in Theorem~\ref{heenenweer}, where
we give such conditions 
in the context of 1-connected regular Lie groups modelled on a barrelled Lie algebra.

\subsection{Derived representations}\label{enekantop}

A representation $(\pi,V)$ of a Lie algebra $\fg$ is a complex vector space $V$ with a Lie algebra 
homomorphism $\pi \colon \fg \rightarrow \mathrm{End(V)}$.
\begin{Definition}{\rm (Unitary Lie algebra representations)}
A representation $(\pi,V)$ of the Lie algebra $\g$ 
is called \emph{unitary} if $V$ is a pre-Hilbert space and
$\pi(\xi)$ is skew-symmetric for all $\xi \in \fg$. 
A \emph{unitary equivalence} between
$(\pi,V)$ and $(\pi',V')$ is
a bijective linear isometry $U \colon V \rightarrow V'$ such that $U \pi(\xi) = \pi'(\xi) U$
for all~$\xi \in \fg$.
We denote the Hilbert completion of $V$ by $\cH_{V}$.
\end{Definition}



If $(\rho,\cH)$ is a smooth unitary representation of a locally convex Lie group $G$, then
its \emph{derived representation} 
$(\dd\rho, \cH^\infty)$ is a unitary representation of its Lie algebra $\fg$. 

\begin{Definition}(Derived representations)\label{derep}
If $(\rho,\cH)$ is a smooth representation of $G$, then
its derived representation
${\dd\rho \: \fg \to \End(\cH^\infty)}$ is defined by
\[
\dd\rho(\xi)\psi :=      
 \frac{d}{dt} \Big\vert_{t=0}\rho(\gamma_t)\psi\,,\]
where $\gamma \colon \R \rightarrow G$ is a curve with 
$\gamma_0 = \one$ and $\frac{d}{dt}\gamma|_{t=0} = \xi$.
\end{Definition}  

\begin{Remark}(Selfadjoint generators)
Suppose that $G$ is regular (or, more generally, that $G$ admits a smooth 
exponential function).
Then we can take $\gamma_t = \exp(t\xi)$ to be the solution of $\delta^{R}_{t}\gamma = \xi$ 
starting in $\1$.
Since the closure of $\dd\rho(\xi)$ coincides with the infinitesimal generator of the 
unitary one-parameter group $t \mapsto \rho(\exp (t\xi))$, the operators
$i\cdot \dd\rho(\xi)$ are then essentially selfadjoint 
(cf.\ \cite[Thm.~VIII.10]{RS75}).
\end{Remark}

The derived representation $\dd\rho$ carries significant information in the sense that 
it determines the restriction of $\rho$ to $G_0$, the connected component of the 
identity. 

\begin{Proposition}{\rm(Infinitesimal representations determine global ones)}
\label{LArulez}
Let $G$ be a connected locally convex Lie group with Lie algebra $\fg$. 
Then $U \colon \cH \rightarrow \widetilde{\cH}$ is an isometric intertwiner
between unitary $G$-representations $(\rho,\cH)$ and  $(\tilde\rho,\tilde\cH)$ 
if and only if $U|_{\cH^{\infty}}$ is an isometric intertwiner between the
$\fg$-representations
$(\dd\rho,\cH^{\infty})$ and $(\dd\tilde\rho,\tilde\cH^{\infty})$.
In particular, $(\rho,\cH)$ is determined by $(\dd\rho, \cH^{\infty})$
up to unitary equivalence.
\end{Proposition}

\begin{proof}
Let $\psi\in \cH^{\infty}$, and let
$\gamma \: [0,1] \to G$ be a smooth curve with 
$\gamma_0 =\1$ and $\gamma_1 = g$. If we define $\xi_t 
:= \delta^{R}_{t}\gamma_t $ in $\fg = T_\1(G)$,
then we have 
\begin{eqnarray*}
\frac{d}{dt} \rho(\gamma_t)\psi &=& \dd\rho(\xi_t) \rho(\gamma_t)\psi\,,\\
\frac{d}{dt} \tilde\rho(\gamma_t)U\psi &=& \dd\tilde\rho(\xi_t) \tilde\rho(\gamma_t)U\psi\,. 
\end{eqnarray*}
Skew-symmetry of the operators $\dd\tilde\rho(\xi_t)$, 
combined with the identity $\dd\tilde\rho(\xi_t)U = U \dd\rho(\xi_t)$ on $\cH^{\infty}$, 
then implies that the 
function 
\[ f(t) := \la \tilde\rho(\gamma_t)U\psi, U\rho(\gamma_t)\psi \ra \] 
is constant because 
\[ f'(t) = 
\la \dd\tilde\rho(\xi_t) \tilde\rho(\gamma_t)U\psi, U\rho(\gamma_t)\psi \ra 
+ \la \tilde\rho(\gamma_t)U\psi, U\dd\rho(\xi_t) \rho(\gamma_t)\psi \ra 
= 0.\] 
Since $f(1)=f(0)$ yields
$\la \tilde\rho(g)U\psi,U\rho(g)\psi \ra = \la \psi,\psi\ra$, 
we have
$U\rho(g)\psi = \tilde\rho(g)U\psi$ for all $\psi \in \cH^{\infty}$
by Cauchy--Schwarz.
As $\cH^{\infty}$ is dense in $\cH$, this yields
$\tilde\rho(g)U = U\rho(g)$ on $\cH$.
\end{proof}

\subsection{Globalisation of Lie algebra representations}\label{anderekantop}

The converse problem, existence and uniqueness of globalisations for a given unitary 
Lie algebra representation, is much harder.
There is a functional analytic and a topological aspect to this question, in the sense that one can ask
whether a given Lie algebra representation integrates to a local group representation, and, if so, 
whether or not this extends to a global representation.
Since the topological part  is less relevant for projective representations,  
we now isolate the functional analytic part.

For the integration problem, we will work in the context of regular Lie groups modelled on 
barrelled Lie algebras.  
\begin{Definition}(Barrelled Lie algebras)
A locally convex Lie algebra $\fg$ is called \emph{barrelled} 
if every barrel
(i.e.\ every closed, convex, circled, absorbing subset) of $\fg$
is a 0-neighbourhood.
\end{Definition}
The distinguishing feature of barrelled spaces is that the Banach--Steinhaus Theorem holds.
Every locally convex Baire space is barrelled, so in particular, all Fr\'echet spaces are barrelled.
Products, Hausdorff quotients and inductive limits of barrelled spaces are barrelled, 
so in particular, LF-spaces are barrelled.

\subsubsection{Weak and strong topology for LA representations}
\label{subsec:3.2}

Let $(\pi, V)$ be a unitary representation of a locally convex Lie algebra $\fg$.
In order to make headway with the integration of Lie algebra representations,
we will need a topology on $V$ for which the map 
$\fg \times V \rightarrow V,  (\xi,\psi) \mapsto \pi(\xi)\psi$ is sequentially continuous.
Since the \emph{norm topology} on $V \subseteq \cH_{V}$, induced by 
the Hilbert space norm on $\cH_{V}$, is too coarse for this to hold
even in the case of finite dimensional Lie algebras,
we introduce two topologies on~$V$, the \emph{weak} and the \emph{strong}
topology, that do fulfil this requirement.

\begin{Remark}(Notation)
The unitary $\fg$-representation $\pi \colon \fg \rightarrow \mathrm{End}(V)$
extends to a $*$-representation
of the universal enveloping 
algebra $\cU_{\C}(\fg)$, equipped with
the unique antilinear anti-involution that 
satisfies $\xi^* = -\xi$ on $\fg \subseteq \mathcal{U}_{\C}(\fg)$.
It will be denoted by the same letter, $\pi \colon \cU_{\C}(\fg) \rightarrow \mathrm{End}(V)$. 
In the same vein, we define the map 
\[ \pi_n \colon \fg^n \rightarrow \mathrm{End}(V), \quad 
\pi_n(\bxi) := \pi(\xi_n)\ldots \pi(\xi_1) \quad \mbox{ for } \quad 
\bxi = (\xi_{n},\ldots ,\xi_{1}) \in \fg^n.\] 
For $n=0$, this is interpreted as
the map on $\fg^0 = \R$ defined by $\pi_0(\lambda) = \lambda \one$, so that
$\pi_n$ is the concatenation of $\pi \colon \cU_{\C}(\fg) \rightarrow \mathrm{End}(V)$
with the $n$-linear multiplication map $\fg^n \rightarrow \cU_{\C}(\fg)$.
%
%
\end{Remark}

\begin{Definition}{\rm (Continuous Lie algebra representations)} \label{LAcontinuity}
We call $(\pi, V)$ \emph{continuous} if for all $n \in \N$ and $\psi \in V$,
the function $\pi_{n}^{\psi} \colon \fg^{n} \rightarrow \cH_{V}$ defined by 
$\pi_{n}^{\psi}(\bxi) := \pi_n(\bxi)\psi$ is continuous w.r.t.\ the 
norm
topology.
\end{Definition}

\begin{Remark}{\rm (Derived representations are continuous)} \label{deriscont}
If $(\rho,\cH)$ is a smooth unitary representation of a locally convex Lie group, then 
$\dd\rho_{n}^{\psi}$ is continuous, as it is the (restriction to $\fg^n$ of the) 
$n^{\mathrm{th}}$ derivative at $\one \in G$ of the 
smooth function $g \mapsto \rho(g)\psi$. It follows that $(\dd\rho,\cH^{\infty})$
is continuous.
\end{Remark}

The intuition behind the definition of the weak and strong topology 
on $V$ is to think of $\psi \in \cH^{\infty}$ as a `state' 
$\sigma_{\psi} \colon \cU_{\C}(\fg) \rightarrow \C$ on $\cU_{\C}(\fg)$, given by
$\sigma_{\psi}(A) = \langle\psi, \pi(A)\psi\rangle$. The weak topology 
of $V$ corresponds precisely to the weak-$*$ topology on $\cU_{\C}^*(\fg)$, 
and the strong topology is loosely based on the
topology of bounded convergence.

\begin{Definition}(Weak and strong topology on $V$)\label{wstopology}
The \emph{weak topology} on $V$ is the locally convex topology generated 
by the seminorms 
\[ p_{\bxi}(\psi) := \|\pi_n(\bxi)\psi\|\quad \mbox{ for } \quad \bxi \in \fg^n, n \in \N.\]
The \emph{strong topology} on $V$ is generated by the seminorms 
\[p_{B}(\psi) := \sup_{\bxi\in B}\|\pi_n(\bxi)\psi\|,\] where $B\subseteq \fg^n$
runs over the bounded subsets of the locally convex spaces~$\fg^n$, and
$n$ runs over $\N$.
\end{Definition}

Recall that for $n=0$, we have $\fg^0 = \R$ and $p_{\lambda}(\psi) = |\lambda|\|\psi\|$, 
so that both the strong and the weak topology are finer than the norm topology
of $V \subseteq \cH_{V}$.
As the name suggests, the strong topology is stronger than the weak topology.
If we have the strong, weak, or norm topology on $V$ in mind, we will call
a convergent sequence in $V$ `strongly', `weakly' or `norm'-convergent, and a
a continuous map 
into $V$ `strongly', `weakly', or `norm'-continuous.

\begin{Remark}{\rm (Comparison with existing definitions)}
For Banach--Lie groups, the strong topology on the set $\cH^{\infty}$
of smooth vectors coincides with the locally convex
topology introduced 
in \cite{Ne10b} for (not necessarily unitary) representations of Banach--Lie groups on locally convex spaces. 
For finite dimensional Lie groups $G$, the weak and the strong topology coincide with each other,
and with the usual topology on $\cH^{\infty}$, derived from the smooth compact open topology on 
$C^{\infty}(G,\cH)$ and the embedding 
$\cH^{\infty} \rightarrow C^{\infty}(G,\cH)$ defined by $\rho^\psi(g) = \rho(g)\psi$ 
(cf.\ \cite[Prop.~4.6]{Ne10b}).
\end{Remark}

\begin{Example}(A smooth representation of $\R^{\N}$)\label{U1example}
Consider the abelian Lie group $G = \R^{\N}$, whose 
Fr\'echet--Lie algebra
$\R^{\N}$ is equipped with the product topology and the trivial Lie bracket. 
The unitary $G$-representation $(\rho, \ell^{2}(\N,\C))$
defined by 
$(\rho(\phi)\psi)_{n} := e^{i\phi_{n}} \psi_{n}$
has the dense space $\cH^{\infty} = \C^{(\N)}$ 
of smooth vectors \cite[Ex.~4.8]{Ne10b}, and the derived Lie algebra representation $(\dd\rho, \C^{(\N)})$
is given by $(\dd\rho(\xi)\psi)_{n} := i \xi_{n}\psi_{n}$.
The weak topology on $\C^{(\N)}$ is the locally convex topology generated by the seminorms 
$p_{\bxi}(\psi) := \sqrt{\sum_{i\in \N} |\bxi_i \psi_{i}|^2}$, where 
$\bxi_i = \xi^{1}_{i} \ldots \xi^{n}_{i}$ if $\bxi = (\xi^1, \ldots, \xi^{n}) \in \fg^n$.
This is precisely the locally convex inductive limit topology on $\C^{(\N)}$, i.e., the strongest locally convex 
topology such that all inclusions $\C^n \hookrightarrow \C^{(\N)}$ are continuous.
Indeed, the opens $B_{\bxi}(0) := \{\psi \in \C^{(\N)}\,;\, p_{\bxi}(\psi) < 1\}$ that generate the weak 
topology are open for the locally convex inductive limit topology.
Conversely,
if $U\subseteq \C^{(\N)}$ is a convex 0-neighbourhood for the locally convex inductive limit topology,
then for each $i\in \N$, there exists an $r_i>0$ such that $\psi \in U$ for all 
$\psi \in \C e_i$ with
$\| \psi\| \leq r_i$. With $\bxi_i := 2^{i+1}/r_i$, the inequality 
$p_{\bxi}(\psi) <1$ implies $|2^{i+1}\psi_i| < r_i$, so that
the multiple $2^{i+1}\psi_i e_i$ of the $i^{\mathrm{th}}$ basis vector $e_i$ 
is an element of $U$. 
Since $\psi = \sum_{i=0}^{n}\psi_i e_i$ is a convex combination of $0$
and the vectors $2^{i+1}\psi_i e_i$,
it is an element of $U$, hence 
$B_{\bxi}(0) \subseteq U$
and $U$ is open for the weak topology.
In this case,
the strong topology yields the same result.
\end{Example}



\begin{Proposition}\label{isookcontinu}
{\rm(Strong continuity $\pi_{n}(\bxi)$)}
For every $\bxi\in \fg^n$,  
the map $\pi_n(\bxi) \: V \rightarrow V$ 
is continuous for the weak as well as for the strong topology.
\end{Proposition}

\begin{proof}
For the weak topology, this follows from
$p_{\bxi'}(\pi_n(\bxi)\psi) = p_{\bxi'\bxi}(\psi)$, 
where $\bxi'\bxi \in \fg^{m+n}$ denotes the concatenation 
of $\xi' \in \fg^m$ and $\xi \in \fg^n$.
For the strong topology, this follows from
$p_{B}(\pi_n(\bxi)\psi) = p_{B \times \{\bxi\}}(\psi)$, as 
$B \times \{\bxi\} \subseteq \fg^{n+m}$ is bounded 
for all $\bxi \in \fg^m$ and all bounded sets 
$B \subseteq \fg^m$.
\end{proof}

Although the norm 
 topology on $V$ is coarser than the weak and the strong topology,
continuity of \emph{all} maps $\pi^{\psi}_{n}\colon \fg^{n} \rightarrow V$ 
w.r.t.\ the norm 
topology implies continuity w.r.t.\ the weak and 
strong topology.

\begin{Proposition}\label{bijdehand}{\rm (Strong continuity $\pi_{n}^{\psi}$)}
If $\pi \colon \fg \rightarrow \mathrm{End}(V)$ is continuous 
in the sense of {\rm Definition~\ref{LAcontinuity}}, 
then for every $\psi \in V$ and $n \in \N$,
the map $\pi^{\psi}_{n}\colon \fg_{\C}^{n} \rightarrow V$
is strongly (hence weakly) continuous.
\end{Proposition}
\begin{proof}
Let $\bxi_0 \in \fg^m$. Then  
norm-continuity of $\pi_{m+n}^{\psi} \colon \fg^{m+n} \rightarrow \cH_{V}$
implies that the map
$\fg^{n} \rightarrow \cH_{V}$ defined by 
$\bxi \mapsto \pi_m(\bxi_0) \pi_n(\bxi)\psi$
is norm-continuous.
Since this is true for all fixed $\bxi_0\in \fg^{m}$, 
$\pi^{\psi}_{n}$ is weakly continuous.

To show that it is even strongly continuous, 
consider the norm $p_{B}$ derived from the bounded set 
$B \subseteq \fg^{m}$.
For every $\varepsilon >0$, there exist open sets
$U_0 \subseteq \fg^m$ and $U \subseteq \fg^n$ such that
$\bxi_0\in U_0$ and $\bxi \in U$ imply $\|\pi_m(\bxi_0)\pi_n(\bxi)\psi\|\leq \varepsilon$.
Since $B \subseteq \lambda U_0$ for some $\lambda > 0$, we have 
$\|\pi_m(\bxi_0)\pi_n(\bxi)\psi\|\leq \lambda^{m}\varepsilon$ 
for $\bxi_0\in B$, $\bxi \in U$. 
Since $\pi_{n}^{\psi}$ is $n$-linear, $\bxi \in \lambda^{-m/n} U$
 implies
 $p_{B}(\pi^{\psi}_{n}(\bxi))\leq \varepsilon$.
\end{proof}

\begin{Lemma}\label{seqcont} {\rm(Sequential continuity of infinitesimal action)}
If $\pi \colon \fg \rightarrow \mathrm{End}(V)$ is a continuous unitary representation of a 
barrelled Lie algebra $\fg$, then the map 
$\fg \times V \rightarrow V, (\xi,\psi) \mapsto \pi(\xi)\psi$
is sequentially continuous for the weak and strong topology.
\end{Lemma}

\begin{proof}
Let $(\xi_n,\psi_n)$ be a sequence in $\fg \times V$ converging weakly to 
$(\xi_{\infty}, \psi_{\infty})$.
The maps $\pi^{\psi_n}  \colon \fg \rightarrow V$ defined by $\pi^{\psi_n}(\xi) := \pi(\xi)\psi_n$
are weakly continuous by Proposition~\ref{bijdehand},
and $H  := \{\pi^{\psi_n}\,;\, n \in \N\}$ is bounded for the topology of pointwise 
convergence because for each fixed $\xi \in \fg$, 
$\pi^{\psi_n}(\xi)$ converges to $\pi(\xi) \psi_{\infty}$ by Proposition~\ref{isookcontinu}.
By the Banach--Steinhaus Theorem (\cite[Thm.~1]{Bo50}), $H$ is equicontinuous,
so that $\lim_{k \rightarrow \infty} \pi(\xi_{\infty} - \xi_{k})\psi_{n} = 0$ uniformly in $n$.
It follows that $\lim_{n \rightarrow \infty} \pi(\xi_{n})\psi_{n} = 
\lim_{n \rightarrow \infty} \pi(\xi_{\infty})\psi_{n}$,
which equals $\pi(\xi_{\infty})\psi_{\infty}$, again by Proposition~\ref{isookcontinu}.

If $(\xi_n,\psi_n) \rightarrow (\xi_{\infty}, \psi_{\infty})$ strongly,
then one considers for each $\bxi\in \fg^{m}$ the map $\pi_{\bxi}^{\psi_n} \colon \fg \rightarrow V$
defined by $\pi_{\bxi}^{\psi_n}(\xi) = \pi_m(\bxi)\pi(\xi)\psi_{n}$, strongly 
continuous by Proposition~\ref{bijdehand}.
For every bounded subset $B\subseteq \fg^{m}$, the set 
\[ H := \{\pi_{\bxi}^{\psi_n}\,;\, n \in \N, \bxi \in B\} \] 
is pointwise bounded because for each fixed $\xi \in \fg$, 
$\|\pi_{\bxi}^{\psi_n}(\xi)\| \leq p_{B\times \{\xi\}}(\psi_{n})$, which, for $n \rightarrow \infty$, approaches $p_{B\times\{\xi\}}(\psi_{\infty})$ by assumption. 
By the Banach--Steinhaus Theorem, $H$ is equicontinuous, and
\[ \lim_{k \rightarrow \infty} p_{B}(\pi(\xi_{\infty} - \xi_{k})\psi_{n}) = 0 \] 
uniformly in $n$.
Since $\lim_{n \rightarrow \infty} \pi(\xi_\infty)\psi_n = \pi(\xi_{\infty})\psi_{\infty}$ 
strongly by Proposition~\ref{isookcontinu},
it follows that $\pi(\xi_{n})\psi_n$ approaches $\pi(\xi_{\infty})\psi_{\infty}$
with respect to every seminorm~$p_{B}$.
\end{proof}

\begin{Remark}(Continuity of $\fg \times V \rightarrow V$)
If $\fg$ is a Banach--Lie algebra, then both $\fg$ and the strong topology on $V$
are first countable, so that sequential strong continuity of 
the action map $(\xi,\psi) \mapsto \pi(\xi)\psi$ implies 
strong continuity. 
However, even for Fr\'echet--Lie algebras $\fg$, the action 
$(\xi,\psi) \mapsto \pi(\xi)\psi$ can be \emph{sequentially} continuous
without being continuous. 
For example, the map $(\xi,\psi) \mapsto \dd\rho(\xi)\psi$ described in Example~\ref{U1example}
is sequentially continuous by Lemma.~\ref{seqcont},
but it is not continuous w.r.t.\ \emph{any} locally convex topology on $V$, cf.\
\cite[Ex.~4.8]{Ne10b}.
\end{Remark}

\begin{Lemma}\label{kromcont} {\rm(Continuous curves)}
Let $\pi \colon \fg \rightarrow \mathrm{End}(V)$ be a continuous unitary representation of a 
barrelled Lie algebra $\fg$, let $t \mapsto \xi_t$ be a continuous curve in $\fg$,
and let and $t \mapsto \psi_t$ be a strongly (weakly)
continuous curve in $V$. Then the curve $t \mapsto \pi(\xi_t)\psi_t$ in $V$
is strongly (weakly) continuous.
\end{Lemma}
\begin{proof}
If $t \mapsto \pi(\xi_t)\psi_t$ were not continuous in some point $t_0$, then it would be possible
to find a sequence $t_n$ converging to $t_0$ such that $\pi(\xi_n)\psi_n$ does not 
converge to $\pi(\xi_0)\psi_0$, contradicting Lemma~\ref{seqcont}.
\end{proof}

The following lemma holds for the weak and strong topology, but not for the norm topology --
which was our main motivation to introduce the former two.
\begin{Lemma}\label{productrule} {\rm ($C^{1}$-curves)}
Suppose that $(\pi,V)$ is a continuous unitary representation of a barrelled Lie algebra $\fg$.
If $t \mapsto \psi_{t}$ is a $C^1$-curve in $V$ for the strong (weak) topology, 
and $t \mapsto \xi_t$ a $C^1$-curve in $\fg$, then 
$t \mapsto \pi(\xi_{t})\psi_t$ is a $C^1$-curve in $V$ for the strong (weak) topology, 
with derivative $\frac{d}{dt} \pi(\xi_{t})\psi_t = \pi(\xi'_{t})\psi_t + \pi(\xi_{t})\psi'_t$.
\end{Lemma}

\begin{proof}
Since $\fg \rightarrow V,  \xi \mapsto \pi(\xi)\psi_t$ 
and $\pi(\xi_t) \: V \rightarrow V$ are continuous w.r.t.\ the strong and weak topology by 
Propositions \ref{bijdehand} and \ref{isookcontinu} respectively,
the first
two terms on the r.h.s.\ of
\begin{eqnarray*}
{\textstyle \frac{d}{dt}} \pi(\xi_{t})\psi_t &=& 
\lim_{h \rightarrow 0}  \pi\left({\textstyle \frac{1}{h}}(\xi_{t+h} - \xi_{t})\right)\psi_{t}
+\lim_{h \rightarrow 0} \pi(\xi_{t})\left({\textstyle \frac{1}{h}}(\psi_{t+h} - \psi_{t})\right)\\
& &
+
\lim_{h \rightarrow 0}  {\textstyle \frac{1}{h}} \pi(\xi_{t+h} - \xi_{t})(\psi_{t+h} - \psi_{t})
\end{eqnarray*}
are $\pi(\xi'_{t})\psi_t$ and $\pi(\xi_{t})\psi'_{t}$.
It remains to show that the third term is zero.
Since $h \mapsto \psi_{t+h} - \psi_t$ is a strongly (weakly) continuous curve in $V$,
and since the map defined by $0 \mapsto \xi'_{t}$ and
$h \mapsto \frac{1}{h}(\xi_{t+h} - \xi_t)$ for $h \neq 0$ is a continuous
curve in $\fg$, Lemma~\ref{kromcont} implies that 
$
\lim_{h \rightarrow 0}  {\textstyle \frac{1}{h}} \pi(\xi_{t+h} - \xi_{t})(\psi_{t+h} - \psi_{t})
= \pi(\xi'_{t})(0) = 0
$.
Since $\frac{d}{dt}\pi(\xi_t)\psi_t = \pi(\xi'_t)\psi_t + \pi(\xi_t)\psi'_t$, strong (or weak)
continuity of the derivative follows directly
from Lemma~\ref{kromcont}.
\end{proof}

The following product rule results from Lemma~\ref{productrule} by induction.
\begin{Proposition}{\rm (Product rule)}\label{moreproducts}
Suppose that $(\pi,V)$ is a continuous unitary representation of a barrelled Lie algebra $\fg$.
Let $t \mapsto \xi_t$ be a $C^n$-curve in $\fg$, and let 
$t \mapsto \psi_{t}$ be a $C^n$-curve in $V$ for the strong (weak) topology, with 
$n \in \N \cup \{\infty\}$.
Then $t \mapsto \pi(\xi_{t})\psi_t$ is a $C^n$-curve in $V$ for the strong (weak) topology, 
and for all $k \leq n$, we have
\[\frac{d^k}{dt^k} \pi(\xi_{t})\psi_t = \sum_{t=0}^{k}  \binom{k}{t} \pi(\xi^{(t)}_{t})\psi^{(k-t)}_t\,.\]
\end{Proposition}

The following completeness result for the space of smooth vectors
will not be needed for the globalisation result towards which we are working, but
is of general interest.

\begin{Proposition}{\rm (Completeness of $\cH^{\infty}$)}\label{completeness}
Let $G$ be a regular Lie group modelled on a Fr\'echet--Lie algebra $\fg$.
Then for any smooth unitary $G$-representation $(\rho,\cH)$, the space 
$\cH^{\infty}$ of smooth vectors is complete for the weak (and hence for the strong) 
topology derived from the $\fg$-representation $(\dd\rho, \cH^{\infty})$.  
\end{Proposition}

\begin{proof}
Suppose that $(\psi_\alpha)_{\alpha \in S}$ is a Cauchy net in 
$\cH^{\infty}$ for the weak topology derived from $(\dd\rho, \cH^{\infty})$
(which is certainly the case if it is Cauchy for the strong topology).
Then for each $\bxi \in \fg^n$, the net $\dd\rho_{n}(\bxi)\psi_{\alpha}$ is Cauchy in $\cH$,
hence converges to some $\psi_{\bxi} \in \cH$ in norm.
We need to show that $\psi_{\one}$ is a smooth vector.
Consider $(\xi_n, \ldots, \xi_1) \in \fg^{n}$, and note that the limits
$\psi_{\one} := \lim \psi_{\alpha}$ and
\begin{equation}\label{smurfendans}
\psi_{\xi_{k}\ldots \xi_{1}} = \lim \dd\rho(\xi_{k}) \ldots \dd\rho(\xi_{1})\psi_{\alpha}
\end{equation}
exist in $\cH$ for $1\leq k \leq n$.
Since $G$ is regular, the map $t \mapsto \rho(\exp(t\xi_k))$ is a continuous 
one-parameter group of unitary operators (w.r.t.\ the strong topology on $\U(\cH)$)
whose selfadjoint generator
$X_{\xi_k}$ restricts to $-i\dd\rho(\xi_k)$ on $\cH^{\infty}$.
In particular, the operators $\dd\rho(\xi_k)$ are closable, so existence 
of the limit~(\ref{smurfendans})
and closedness of $iX_{\xi_{k+1}}$
implies $\psi_{\xi_{k+1}\ldots \xi_{1}} = iX_{\xi_{k+1}} \psi_{\xi_{k}\ldots \xi_{1}}$. 
One then proves by induction that $\psi_{\xi_{n}\ldots \xi_{1}} = i^{n}X_{\xi_n}\ldots X_{\xi_{1}}\psi_{\one}$.

In order to prove that $\psi_{\one}\in \cH^{\infty}$, we show that the map 
\[ \Omega^{\psi_{\one}} \colon \fg^n \rightarrow \cH, \quad 
\Omega^{\psi_{\one}}(\xi_n,\ldots, \xi_1) := i^nX_{\xi_n}\ldots X_{\xi_1}\psi_{\one}\] is 
$n$-linear and continuous.
It is linear in each $\xi_{k}$ because 
\begin{eqnarray*}
\psi_{\xi_{n}\ldots (s\xi_k + t \xi'_k)\ldots  \xi_1} &=& \lim_{\alpha\in S} 
s \dd\rho(\xi_{n} \ldots \xi_{k} \ldots \xi_1)\psi_\alpha + 
\lim_{\alpha\in S} t \dd\rho(\xi_{n} \ldots \xi'_{k} \ldots \xi_1)\psi_\alpha
\end{eqnarray*}
is the sum of two nets converging to 
$s \psi_{\xi_{n}\ldots \xi_{k}\ldots  \xi_1}$ and $t\psi_{\xi_{n}\ldots \xi_{k}' \ldots  \xi_1}$.
We show that $\Omega^{\psi_{\one}}$ is separately continuous 
by means of a uniform boundedness argument.
Let $u_{\alpha} \colon \fg \rightarrow \cH$ be defined by
\[u_{\alpha}(\xi) := \dd\rho(A)\dd\rho(\xi)\dd\rho(B)\psi_{\alpha}\quad \mbox{ where } \quad 
A = \xi_n \ldots \xi_{k+1}, B = \xi_{k-1}\ldots \xi_1 \] 
are fixed elements 
of $\cU_{\C}(\fg)$.
Since $\psi_{\alpha} \in \cH^{\infty}$, continuity of $(\dd\rho,\cH^{\infty})$ (cf.\ Remark~\ref{deriscont})
implies that the $u_{\alpha}$ are continuous, and the set $H := \{u_{\alpha} \,;\, \alpha \in S\}$
is pointwise bounded because for each fixed $\xi \in \fg$, the net 
$u_{\alpha}(\xi)$ is norm-convergent.  
Since $\fg$ is barrelled,
the Banach--Steinhaus Theorem \cite[Thm.~1]{Bo50} implies that $H$
is equicontinuous;
for every $\varepsilon >0$, there exists a 0-neighbourhood $U \subseteq \fg$
such that $\xi\in U$ implies $\|\dd\rho(A)\dd\rho(\xi)\dd\rho(B)\psi_{\alpha}\| \leq \varepsilon$ uniformly for
all $\alpha \in S$.
It follows that for $\xi \in U$, also the limit for $\alpha \in S$ satisfies $\|\psi_{A\cdot \xi\cdot B}\|\leq \varepsilon$.


Since $\fg$ is a Fr\'echet space, separate continuity of the $n$-linear map $\Omega^{\psi_{\one}}$
implies joint continuity \cite[Thm.~2.17]{Ru91}, so we may conclude 
from \cite[Lemma 3.4]{Ne10b} that $\psi_{\one}$ is a smooth vector, 
hence $\cH^{\infty}$ is complete. 
\end{proof}

\subsubsection{Regular Lie algebra representations}
Let $G$ be a regular Lie group modelled on a barrelled Lie algebra $\fg$.
We will show that a unitary Lie algebra representation $(\pi, V)$ of 
$\fg$ globalises to a smooth representation of the universal cover $\widetilde{G}_0$
of the connected $\one$-component $G_0$ of $G$ if and only if it is 
weakly or strongly \emph{regular}
in the following sense. 

\begin{Definition}(Regular representations)
We call a unitary Lie algebra representation $(\pi,V)$ 
{\it strongly (weakly) regular}
if it is continuous (cf.~Definition~\ref{LAcontinuity}) 
and if for every smooth function $\xi \colon [0,1]\rightarrow \fg$, the differential equation
\begin{equation}\label{ODE}
{\textstyle \frac{d}{dt}}\psi_t = \pi(\xi_t) \psi_t
\end{equation}
with initial condition $\psi|_{t=0} = \psi_0$
has a smooth solution in $V$ whose value in $1$ depends smoothly on the path $\xi \in C^{\infty}([0,1],\fg)$, 
where $V$ is equipped with the strong (weak) topology
(cf.\ Definition~\ref{wstopology}).
\end{Definition}



\begin{Remark}(Uniqueness of solutions)\label{inproduct}
Solutions to (\ref{ODE}) are automatically unique if they exist;
if both $\psi_{t}$ and $\tilde{\psi}_{t}$ are solutions, then by
continuity (hence smoothness) of $\langle\,\cdot\,,\,\cdot\,\rangle$
and skew-symmetry of $\pi(\xi_{t})$, we have 
\[\frac{d}{dt}\langle \psi_t, \tilde{\psi}_t\rangle =  \langle \pi(\xi_{t}) \psi_{t},\tilde{\psi}_{t} \rangle 
+ \langle \psi_{t}, \pi(\xi_{t}) \tilde{\psi}_{t} \rangle = 0.\]
Since $\|\psi_{t}\|^2 = \|\tilde{\psi}_{t}\|^2 = \langle \psi_t, \tilde{\psi}_t\rangle$ holds
for $t=0$, it holds for all $t$, hence $\psi_t = \tilde{\psi}_t$.
\end{Remark}

We first prove that all derived representations are regular.
The first step is the following lemma,
which reaches a slightly stronger statement than 
Proposition~\ref{bijdehand} if $\pi$ is derived from a group representation.

\begin{Lemma}{\rm (Weak and strong continuity of orbit maps)}\label{bijdehandtwee}
If $(\rho, \cH)$ is a smooth unitary representation 
of a locally convex Lie group $G$, then,  for every $\psi \in \cH^{\infty}$,
the map $\rho_{n}^{\psi} \colon \fg^n \times G \rightarrow \cH^{\infty}$
given by $\rho^{\psi}_{n}(\bxi,g) := \dd\rho_n(\bxi)\rho(g)\psi$
is both weakly and strongly  continuous.
\end{Lemma}

In order to prove the `strong' part of 
Lemma~\ref{bijdehandtwee}, we need the following proposition.

\begin{Proposition}{\rm ($C^1$-estimate)}\label{Helgeslemma}
Let $X$ and $E$ be locally convex vector spaces, 
$U \subseteq X$ open, $k \in \N$, and
$f \colon U \times E^{k} \rightarrow \cH$ a $C^1$-map such that 
$f(u, \,\cdot\,) \colon E^{k} \rightarrow \cH$ is $k$-linear for each $u\in U$.
Then for each $x \in U$, there exists a continuous seminorm
$\|\,\cdot\,\|_{p}$ on $X$  with $B^1_{p}(x) \subseteq U$, 
and a continuous seminorm 
$\|\,\cdot\,\|_{q}$ on $E$ such that
\begin{eqnarray}\label{eekhoorn}
\|f(y, \xi_1, \ldots, \xi_k) - f(x, \xi_1, \ldots, \xi_{k})\| \leq 
\|y-x\|_{p}\|\xi_1\|_{q}\ldots\|\xi_{k}\|_{q} 
\end{eqnarray}
for all $y \in B^1_{p}(x)$ and $(\xi_1, \ldots, \xi_{k}) \in E^{k}$.      
\end{Proposition}
\begin{proof} This proposition, as well as its proof, is taken from \cite[Lemma~1.6(b)]{Gl07}.
Since $Df \colon (U \times E^k) \times (X \times E^k) \rightarrow \cH$ is continuous
and $Df(x,0,0,0) = 0$, there exists a seminorm $p$ on $X$
such that the unit ball $B^{p}_{1}(x)$ around $x$ is contained in $U$,
and a seminorm $q$ on $E$
such that 
\begin{equation}\label{konijnevel}
\|Df(y, \xi_1, \ldots, \xi_k, z, \eta_1, \ldots, \eta_k)\| \leq 1
\end{equation}  
for all
$y \in B^{p}_{1}(x) \subseteq U$ and  
$z \in B^{p}_{1}(0) \subseteq X$, and for
$\xi_j, \eta_j \in B^{q}_{1}(0) \subseteq E$ with $1 \leq j \leq k$.
By linearity in $z$, this 
implies
\begin{equation}\label{muizenplaagje}
\|Df(y, \xi_1, \ldots, \xi_k, z, 0, \ldots, 0)\| \leq \|z\|_{p}
\end{equation}
for $y \in B^{p}_{1}(x)$, $\xi_{j} \in B^{q}_{1}(0)$ and $z \in X$.
We now use the Mean Value Theorem to write
\[
f(y,\xi_1, \ldots \xi_k) - f(x, \xi_1, \ldots \xi_k) = 
\int_{0}^{1}Df(x + t(y-x), \xi_1, \ldots, \xi_k, y-x, 0, \ldots, 0) dt
\]
for $y \in B^{p}_{1}(x)$ and $\xi_j \in B^{q}_{1}(0)$.
By (\ref{muizenplaagje}), the integrand is bounded by 
$\|y - x\|_{p}$ in norm, so we obtain
\[
\|f(y,\xi_1, \ldots \xi_k) - f(x, \xi_1, \ldots \xi_k)\| \leq \|y - x\|_{p}
\]
for $y \in B^{p}_{1}(x)$, $\xi_{j} \in B^{q}_{1}(0)$.
Equation~(\ref{eekhoorn}) now follows from $k$-linearity of the map
${f(y,\,\cdot\,) - f(x,\,\cdot\,)} \colon E^{k} \rightarrow \cH$\,.
\end{proof} 

We proceed with the proof of Lemma~\ref{bijdehandtwee}.

\begin{proof}
We first prove weak continuity:
since the map $\rho_{m+n}^{\psi}$ is norm-continuous,
so is the map $\dd\rho_m(\bxi_{0})\rho^{\psi}_{n}$ for every fixed $\bxi_0 
\in \fg^{m}$.
This implies that $\rho^{\psi}_{n}$ is continuous for every norm $p_{\bxi_0}$, hence
weakly continuous.

To prove the strong continuity, we need to exhibit, for 
every $(\bxi,g) \in \fg^{n}\times G$,  
every bounded 
set $B \subseteq \fg^{m}$ and every $\varepsilon >0$,
an open neighbourhood $W \subseteq \fg^{n}\times G$ of $(\bxi,g)$ such that 
for all $(\bxi', g') \in W$, we have 
$p_{B}(\rho^{\psi}_{n}(\bxi',g') - \rho^{\psi}_{n}(\bxi,g)) \leq \varepsilon$.

We apply Lemma~\ref{Helgeslemma}
with $k = m$, $E = \fg$, $X = \fg^{n} \times \fg$, and $U \subseteq \fg^{n} \times \fg$ 
a coordinate patch corresponding with an open neighbourhood $U'$ of $(\bxi,g)$ in 
$\fg^{n} \times G$, for which $(\bxi,g) \in U'$ 
is represented by $x \in U$.
Considering the restriction of the smooth map
$\rho^{\psi}_{n+m} \colon \fg^{m}\times \fg^{n} \times G \rightarrow \cH$
to $\fg^{m} \times U'$ and identifying $U'$ with $U$, we find continuous seminorms 
${\|\,\cdot\,\|_{q}}$ on $\fg$ and ${\|\,\cdot\,\|_{p}}$ on $\fg^{n}\times \fg$ 
such that for $\bxi_0 = (\xi_1,\ldots, \xi_{m}) \in \fg^{m}$ and $(\bxi',g') \in U'$, we have
\[
\|\dd\rho_{m}(\bxi_0)\dd\rho_{n}(\bxi')\rho(g')\psi - \dd\rho_{m}(\bxi_0)\dd\rho_{n}(\bxi)\rho(g)\psi\|
\leq
\| x'- x\|_{p}
\|\xi_{1}\|_{q}\ldots \|\xi_{m}\|_{q}
\,,\\
\]
where $x' \in U$ is the coordinate of $(\bxi',g')$.
Let $B^{q}_{1}(0) \subseteq \fg$ be the open unit ball w.r.t.\ the seminorm ${\|\,\cdot\,\|_{q}}$,
and let $U_0 := B_1^{q}(0)^m \subseteq \fg^{m}$. Since $B$ is bounded, we can choose $\lambda \geq 1$
such that $B \subseteq \lambda U_0$.
Then $\bxi_0 \in B$ implies that $\widetilde{\bxi}_0 := \lambda^{-1}\bxi_0 \in U_0$, so that,
since $\dd\rho_{m}$ is $m$-linear,
\begin{eqnarray*}
&&\|\dd\rho_{m}(\bxi_0)\dd\rho_{n}(\bxi')\rho(g')\psi - \dd\rho_{m}(\bxi_0)\dd\rho_{n}(\bxi)\rho(g)\psi\|\\
&=&
\lambda^{m} \|\dd\rho_{m}(\widetilde{\bxi}_0)\dd\rho_{n}(\bxi')\rho(g')\psi -
 \dd\rho_m(\widetilde{\bxi}_0)\dd\rho_{n}(\bxi)\rho(g)\psi\|
 \leq  
\lambda^{m}\| x'- x\|_{p}
\end{eqnarray*}
for $\bxi_0 \in B$, $(\bxi', g') \in U'$.
Choosing the open neighbourhood $W \subseteq U'$ corresponding to 
the open ball $B^{p}_{\lambda^{-m}\varepsilon}(x) \subseteq U$
around $x$ with radius $\lambda^{-m}\varepsilon$ w.r.t.\ $\|\,\cdot\,\|_{p}$, 
we have  
\[
\|\dd\rho_m(\bxi_0)\dd\rho_n(\bxi')\rho(g')\psi 
- \dd\rho_m(\bxi_0)\dd\rho_n(\bxi)\rho(g)\psi\|
\leq
\varepsilon\\
\]
uniformly for $\bxi_0 \in B$, $(\bxi', g') \in W$, so that, 
$p_{B}(\rho_{n}^{\psi}(\bxi',g') - \rho_{n}^{\psi}(\bxi',g')) \leq \varepsilon$ 
for $(\bxi', g') \in W$ as desired.
\end{proof}

Using Lemma~\ref{bijdehandtwee}, we now prove that the orbit map is smooth.

\begin{Lemma}{\rm(Smooth orbit maps)}\label{smoothorbit}
If $(\rho,\cH)$ is a smooth unitary representation of a locally convex Lie group $G$,
then for every $\psi \in \cH^{\infty}$, 
the orbit map ${\rho^{\psi}\colon G \rightarrow \cH^{\infty}}$, defined by 
$\rho^{\psi}(g) := \rho(g)\psi$,
is smooth for the weak and strong topology on~$\cH^{\infty}$.
\end{Lemma}
\begin{proof}
The orbit map is smooth for the norm 
topology on $\cH^{\infty}$ by definition, so
for each $\psi \in \cH^{\infty}$, the 
maps $\fg^m \times \fg^n \times G \rightarrow \cH^{\infty}$
defined by $(\bxi_0,\bxi,g) \mapsto \dd\rho_m(\bxi_0)\dd\rho_n(\bxi)\rho(g)\psi$
are smooth, 
as the restrictions to $\fg^m \times \fg^m \times G$ of the 
$(m+n)^{\mathrm{th}}$ derivative of the orbit map.
It follows from this that the maps 
\begin{equation}\label{horizontaal}
\rho_{n}^{\psi} \colon \fg^n \times G \rightarrow \cH^{\infty}, 
\quad \rho_{n}^{\psi}(\bxi,g) := \dd\rho_n(\bxi)\rho(g)\psi
\end{equation}
are $C^1$ for the weak topology; for fixed $\bxi_0 \in \fg^m$ and for every smooth curve
$t \mapsto (\bxi_t,\gamma_t)$ in $\fg^n \times G$, 
we have 
\begin{eqnarray*}
& &\lim_{h \rightarrow 0} \dd\rho_m(\bxi_0)
	{\textstyle{\frac{1}{h}}}
	\big(\rho_{n}^{\psi}(\bxi_{t+h},\gamma_{t+h}) -  \rho_{n}^{\psi}(\bxi_{t},\gamma_{t})\big)\\
	& &\qquad
	= \dd\rho_m(\bxi_0) \left(\dd\rho_n(\bxi'_{t})\rho(\gamma_t)\psi 
	+ \dd\rho_n(\bxi_t)\dd\rho(\delta^{R}_{t}\gamma)\rho(\gamma_t)\psi\right)
\end{eqnarray*}
w.r.t.\ the norm 
topology, where,
if $\bxi_t = (\xi_1(t), \ldots, \xi_{n}(t))$, the expression $\bxi'_t$ should be understood as 
\[\bxi'_{t} = \big(\xi'_{1}(t), \xi_2(t), \ldots, \xi_{n}(t)\big) + \ldots + 
\big(\xi_{1}(t), \ldots, \xi_{n-1}(t),\xi'_{n}(t)\big)\,.\]
Therefore, the directional derivatives w.r.t.\ the norms $p_{\bxi_0}$ exist and are linear combinations of 
maps of the type $\rho_{n}^{\psi}$ and $\rho^{\psi}_{n+1}$. 
Since these are continuous in the weak topology 
by Lemma~\ref{bijdehandtwee}, the desired $C^1$-property follows.

For $n=0$, this shows that $\rho^{\psi} \colon G \rightarrow \cH^{\infty}$ is $C^1$ for the weak topology,
and that $D\rho^{\psi} \colon TG \rightarrow \cH^{\infty}$ is the map $\rho^{\psi}_1$ under 
the identification $TG \simeq G \times \fg$.
If $\rho^{\psi}$ is $C^n$ for the weak topology and its $n^{\mathrm{th}}$ derivative corresponds to a sum of
maps of type $\rho^{\psi}_{r}$ under 
the identification $T^{(n)}G \simeq G \times \prod_{k=0}^{n-1}\fg^{2^{k}}$, then the same is true for 
$n+1$ because maps of type $\rho_{r}^{n}$ are $C^1$ with derivatives of the same type.
By induction, it follows that the orbit map is smooth for the weak topology.

For the strong topology, fix a smooth curve $t \mapsto (\bxi_t, \gamma_t)$ in $\fg^{n} \times G$,
and use Taylor's Theorem to see that for a bounded set $B \subseteq \fg^m$, the difference
\begin{align*}
\Delta_{h}  :=  \sup_{\bxi_0 \in B} \Big\|&\dd\rho_m(\bxi_0)
{\textstyle\frac{1}{h}}\big(\dd\rho_n(\bxi_{t+h})\rho(\gamma_{t+h}) - \dd\rho_n(\bxi_t)\rho(\gamma_t)\big)\psi\\
& 
- \dd\rho_m(\bxi_0)\big(\dd\rho_n(\bxi'_t) + \dd\rho_n(\bxi_t)\dd\rho(\delta^{R}_{t}\gamma_t)\big)\rho(\gamma_t) 
\psi\Big\|
\end{align*}
for $h\neq 0$
is bounded by the remainder term involving the second derivative,
\[\Delta_h \leq |h| \sup\{ C_{\bxi_0,s} \,;\, \bxi_0\in B, s \in [-h,h]\}\,,\] 
with $C_{\bxi_0,s}$ the norm of the second derivative at $t+s$,
\begin{eqnarray*} 
C_{\bxi_0,s} &:=& \Big\|  
\dd\rho_m(\bxi_0)\Big(
\dd\rho_n(\bxi''_{t+s})+
2\dd\rho_n(\bxi'_{t+s})\dd\rho(\delta^{R}_{t+s}\gamma)\\ & & +
\dd\rho_n(\bxi)\dd\rho(\delta^{R}\gamma'_{t+s}) + 
\dd\rho_n(\bxi) \dd\rho(\delta^R\gamma_{t+s})^2 \Big)
\rho(\gamma_{t+s})\psi
\Big\|\,.
\end{eqnarray*}
Since the map $(\bxi_0, s) \mapsto C_{\bxi_{0}, s}$ is continuous in $\bxi_0$ and $s$,
and positively homogeneous in $\bxi_0$, 
there exists an $\varepsilon >0$ such that $C_{\bxi_{0}, s}$ is bounded
on $B \times [-\varepsilon,\varepsilon]$, say by $C>0$.
Since $\Delta_{h} \leq |h|\cdot C$ uniformly for $\bxi \in B$ (subject to $|h| < \varepsilon$),
we see that the directional derivative
\[\frac{d}{dt}\dd\rho_n(\bxi_t)\rho(\gamma_t)\psi = 
\big(\dd\rho_n(\bxi'_t) + \dd\rho_n(\bxi_t)\dd\rho(\delta^R\gamma_{t})\big)\rho(\gamma_t)\psi\]
exists for the strong topology, and that the result is continuous by Lemma~\ref{bijdehandtwee}. 
It follows that the maps $\rho_{n}^{\psi}$ are $C^1$ for the strong topology.
The proof that the orbit map is smooth for the strong topology is then the same as 
for the weak topology.
\end{proof}

\begin{Proposition}{\rm (Derived representations are regular)}\label{derivedisregular}
Let $(\rho,\cH)$ be a
smooth unitary representation of a regular locally convex Lie group $G$. 
Then the derived representation $(\dd\rho, \cH^{\infty})$ is strongly (hence weakly) regular.
\end{Proposition}

\begin{proof}
The derived representation is continuous by Remark~\ref{deriscont}, so
it remains to show that equation (\ref{ODE}) has a smooth solution $t \mapsto \psi_t$
that depends smoothly on the path $\xi \in C^{\infty}([0,1],\fg)$.
Since $G$ is regular, the ODE 
$\delta^{R}_{t}\gamma_t = \xi_t$ with $\gamma_0 = \one$ has a smooth solution 
$t \mapsto \gamma_t$ in $G$, which
depends smoothly on the path $\xi$. 
Since the orbit map $\rho^{\psi}\colon G \rightarrow \cH^{\infty}$ defined by 
$\rho^{\psi}(g) = \rho(g)\psi$ is smooth w.r.t.\ the weak as well as the strong
topology (Lemma~\ref{smoothorbit}), the path $\psi_t = \rho(\gamma_t)\psi$
is smooth in $\cH^{\infty}$ for the weak and strong topology, 
with derivative $\dd\rho(\delta^{R}_{t}\gamma)\psi_t$
as desired. Since $\gamma_1$ depends smoothly on the path $t \mapsto \xi_t$
and $\rho(\gamma_1)\psi$ depends smoothly on $\gamma_1$, the derived representation is regular. 
\end{proof}

\begin{Proposition}\label{regularisderived} {\rm (Regular implies integrable)}
If $G$ is a 1-connected Lie group modelled on a barrelled Lie algebra $\fg$,
 then every unitary $\fg$-representation
$(\pi,V)$ 
that is either weakly or strongly regular
integrates to a smooth unitary $G$-representation on the Hilbert completion 
$\cH_{V}$ of $V$,
and $V \subseteq \cH_{V}$ is $G$-invariant.
\end{Proposition}
\begin{proof}
For $\psi \in V$, we wish to define
$\rho(g)\psi$ as the solution of $\frac{d}{dt} \psi_{t} = \pi(\delta^{R}_{t}g_t)\psi_t$ with 
$\psi_0 = \psi$, where $t \mapsto g_t$ is a smooth path in $G$ 
with $g_0 = \one$, $g_1 = g$, and `sitting instants'.
We need to show that this definition is independent of the choice of path.

First of all,
since $G$ is simply connected, any pair $(g^0_{t},g^1_{t})$ of such paths
is connected by an
endpoint preserving 
homotopy $\gamma \colon [0,1]^2 \rightarrow G$. 
We show that $\gamma$ can be chosen to be smooth.
Indeed, since the paths have sitting instants, 
$g^1(g^0)^{-1}$ is an element of the Lie group  $C_{*}^{\infty}(\bS^1,G)$ 
of smooth based paths in~$G$. 
Since $G$ is simply connected, $C_{*}(\bS^1,G)$ is connected, 
hence $C_{*}^{\infty}(\bS^1,G)$ is connected by
\cite[Thm.~A.3.7]{Ne02}.
We can thus find a 
smooth path $\Gamma \colon [0,1] \rightarrow C_{*}^{\infty}(\bS^1,G)$
from $\one$ to $g^1(g^0)^{-1}$. This
yields an endpoint preserving homotopy $\gamma(s,t) := \Gamma(s)_{t}g^0_{t}$, which is smooth by
\cite[Prop.~12.2(a)]{Gl04}.
 
We now use the Maurer--Cartan equation to show that the solution to 
\begin{equation}\label{evolution}
\partial_{t}\psi^{s}_{t} = \pi(\delta^{R}_{t}\gamma^{s}_{t})\psi^{s}_{t}
\end{equation}
with initial value $\psi^{s}_{0} = \psi$
satisfies $\partial_{s}\psi^{s}_{1} = 0$, so that the endpoint $\psi^{s}_{1}$ 
does not depend on $s$.
Since the smooth path $[0,t] \rightarrow G$ defined by $\tau \mapsto \gamma^{s}_{\tau}$ depends
smoothly on $t$ and $s$, so does the $\fg$-valued smooth path $[0,t] \rightarrow \fg$ given by
$\tau \mapsto \delta^{R}_{\tau}\gamma^{s}_{\tau}$. 
By regularity of $\pi$, 
the solution $\psi_{t}^{s}\in V$ of (\ref{evolution}) also varies smoothly in $s$ and $t$.
Using the Maurer--Cartan equation~(\ref{MChammer}), the evolution equation (\ref{evolution}) and the 
Product Rule
\ref{moreproducts} (which uses that $\fg$ is barrelled), 
we obtain
\begin{eqnarray*}
\partial_s \partial_{t} \psi_{t}^{s} &=& 
\pi(\partial_{s}\delta^{R}_{t}\gamma^{s}_{t})\psi_t^s + 
\pi(\delta^{R}_{t}\gamma^{s}_{t})\partial_{s}\psi^{s}_{t}\\
&=&
\pi(\partial_{t}\delta^{R}_{s}\gamma^{s}_{t})\psi^{s}_t 
+ \pi([\delta^{R}_{s}\gamma^{s}_{t}, \delta^{R}_{t}\gamma^{s}_{t}])\psi^{s}_{t}
+ 
\pi(\delta^{R}_{t}\gamma^{s}_{t})\partial_{s}\psi^{s}_{t}\\
&=&
\big(\pi(\partial_{t}\delta^{R}_{s}\gamma^{s}_{t})\psi^{s}_t 
+\pi(\delta^{R}_{s}\gamma^{s}_{t})\pi(\delta^{R}_{t}\gamma^{s}_{t})\psi^{s}_{t}
\big)\\
& &
+ 
\big(
\pi(\delta^{R}_{t}\gamma^{s}_{t})\partial_{s}\psi^{s}_{t}
- \pi(\delta^{R}_{t}\gamma^{s}_{t})\pi(\delta^{R}_{s}\gamma^{s}_{t})\psi^{s}_{t}\big)\\
&=&
\partial_{t}\big(\pi(\delta^{R}_s \gamma^{s}_{t})\psi^{s}_{t}\big)
+ 
\pi(\delta^{R}_{t}\gamma^{s}_{t})
\big(\partial_{s}\psi^{s}_{t}
- \pi(\delta^{R}_{s}\gamma^{s}_{t})\psi^{s}_{t}\big)\,
\end{eqnarray*}
(derivatives are w.r.t.\ the strong or weak topology),
so, since $\partial_{t}\partial_{s}\psi_{t}^{s} = \partial_{s}\partial_{t}\psi_{t}^{s}$, we find
\[
\partial_{t}\big(  \partial_{s}\psi^{s}_{t}
- \pi(\delta^{R}_{s}\gamma^{s}_{t})\psi^{s}_{t}\big)
=
\pi(\delta^{R}_{t}\gamma^{s}_{t})
\big(\partial_{s}\psi^{s}_{t}
- \pi(\delta^{R}_{s}\gamma^{s}_{t})\psi^{s}_{t}\big)\,.
\]
Since both $\psi^{s}_{0} = \psi$ and $\gamma^{s}_{0} = \one$ are constant,
$\partial_{s}\psi^{s}_{t}
- \pi(\delta^{R}_{s}\gamma^{s}_{t})\psi_{t}^s$
satisfies the evolution equation~(\ref{evolution}) with initial condition zero, 
hence is identically zero by uniqueness (cf.\ Remark~\ref{inproduct}).
We thus have 
\begin{equation}
\partial_{s}\psi^{s}_{t}
= \pi(\delta^{R}_{s}\gamma^{s}_{t})\psi^{s}_{t}\,,
\end{equation}
so that in particular $\partial_{s}\psi^{s}_{1} = 0$ and  
$\rho(g)\psi$ is well defined.

The maps $\rho(g)$ are linear and unitary (cf.\ Remark~\ref{inproduct}), and 
$g \mapsto \rho(g)\psi$ is smooth by regularity of $G$ and $\pi$.
To show that $\rho(gh) =\rho(g)\rho(h)$, we observe that the concatenation $c_{t}$
of the paths $h_{2t}$ and $g_{2t-1}h_1$ is a
smooth path from $\one$ to $gh$ if we choose $g_t$ and $h_t$ to have `sitting instants'.
The solution of the evolution equation $\frac{d}{dt}\psi_t = \pi(\delta^{R}_{t}c_t)\psi_t$
is then given by $\psi_t = \rho(h_{2t})$ for $t \in [0,\frac{1}{2}]$ and $\psi_{t} = \rho(g_{2t-1})\rho(h_1)\psi$
for $t\in [\frac{1}{2},1]$, so evaluation in $t=1$ yields $\rho(gh)\psi = \rho(g)\rho(h)\psi$.
\end{proof}

Combining Propositions \ref{LArulez}, \ref{derivedisregular} and \ref{regularisderived}, 
we obtain the following characterisation of smooth unitary representations.

\begin{Theorem}{\rm (Correspondence global/infinitesimal representations)}
\label{heenenweer}
Let $G$ be a $1$-connected regular Lie group modelled on a barrelled Lie algebra~$\fg$.
Then every smooth unitary representation $(\rho,\cH)$ of $G$ 
gives rise to a derived Lie algebra representation $(\dd\rho,\cH^{\infty})$
which is both weakly and strongly regular.
Conversely, every unitary $\fg$-representation $(\pi,V)$ that is either weakly or strongly regular
integrates to a smooth unitary $G$-representation.
Two $G$-representations are
unitarily equivalent if and only if their derived $\fg$-representations are unitarily equivalent.
\end{Theorem}

\begin{Corollary} {\rm(Equivalence weak and strong regularity)}
Let $G$ be a regular Lie group modelled on a barrelled Lie algebra $\fg$.
Then a unitary $\fg$-representation $(\pi,V)$ is weakly regular if and only if it is 
strongly regular. 
\end{Corollary}
\begin{proof}
If $(\pi,V)$ is either weakly or strongly regular, then it integrates to a smooth $G$-representation,
so it is both weakly and strongly regular by Proposition~\ref{derivedisregular}.
\end{proof}

\section{Projective unitary representations and central extensions}
\label{sec:4}

Every smooth unitary representation of a locally convex Lie group $G$ 
gives rise to a smooth projective unitary representation, and we will call 
a projective representation \emph{linear} if it is of this form.
Although not every smooth projective unitary representations of $G$ is linear,
we will now show that it can always be viewed as a smooth linear 
representation for a 
\emph{central extension} of $G$ by the circle group $\T$. 
We prove that 
smooth \emph{projective} unitary representations of a locally convex
Lie group $G$ give rise to smooth \emph{linear} unitary representations of a 
central Lie group 
extension $G^{\sharp}$ of $G$, and hence to a representation of the central Lie algebra 
extension $\fg^{\sharp}$ of $\fg$. Just like in the finite dimensional case, 
no information is lost in this transition if $G$ is connected.


\begin{Definition}(Central Lie group extensions)\label{GroepUitbreiding}
A \emph{central extension} of $G$ by $\T$ is 
an exact sequence
\[1 \rightarrow \T \rightarrow \widehat{G} \rightarrow G \rightarrow 1\]
of locally convex Lie groups such that the image of $\T$ is central in $\widehat{G}$
and $\widehat{G} \rightarrow G$ is a locally trivial principal $\T$-bundle.
An \emph{isomorphism} $\Phi \colon \widehat{G} \rightarrow \widehat{G}'$ of
central $\T$-extensions is an isomorphism of locally convex Lie groups   
that induces the identity maps on $G $ and $\T$.
\end{Definition}

Given a continuous projective unitary representation $(\ol\rho,\cH)$ of $G$,
the morphism 
\[ G^\sharp := \{ (g,U) \in G \times \U(\cH) \: \oline\rho(g) = \oline U \} \to G, \quad
(g,U) \mapsto g \] 
defines a central $\T$-extension of \emph{topological} groups. Its continuous unitary 
representation 
\[\hat\rho \: G^\sharp \to \U(\cH), \quad (g,U) \mapsto U\] 
reduces to $z \mapsto z\one$ on $\T$, hence
induces $\overline{\rho}$ on $G$.
To show that $G^{\sharp}$ is a locally convex Lie group if $(\ol\rho,\cH)$ is smooth,
we need the notion of (local) \emph{lifts}.

\begin{Definition}(Lifts and cocycles) \label{def:1.3} 
Let $(\ol\rho,\cH)$ be a smooth projective unitary representation of $G$.
A \emph{lift} of $\ol\rho$ over a symmetric $\one$-neighbourhood $U_{\one} \subseteq G$ 
is a function $\rho \: U_{\one} \to \U(\cH)$ with $\rho(\1) = \1$ and
$\oline{\rho(g)} = \oline\rho(g)$ for every $g \in U_\1$.
If $V_{\one}\subeq U_\1$ is a symmetric $\one$-neighbourhood such that $g,h \in V_{\one}$ implies $gh \in U_{\one}$,
then the local \emph{2-cocycle} 
$f \: V_{\one} \times V_{\one} \to \T$, defined by
\begin{equation}
  \label{eq:cocycl}
 \rho(g) \rho(h) = f(g,h) \rho(gh) \quad \mbox{ for all } \quad g,h \in V_\1, 
\end{equation}
measures the failure of $\rho$ to be a local group homomorphism.
\end{Definition}

If two lifts
$\rho$ and $\rho'$ differ by a smooth function (a local 1-cocycle) 
$\beta \colon U_{\one} \rightarrow \T$ (and this will be the case for the 
lifts considered below),
the corresponding local 2-cocycles
$f$ and $f'$ differ by the local coboundary $\delta \beta \colon V_{\one}\times V_{\one} \rightarrow \T$,
defined by 
$\delta\beta(g,h):= \beta(g)\beta(h)\beta(gh)^{-1}$. Note that $\beta(\one) = 1$, 
and $f(g,\1) = f(\1,g) = 1$ for all $g \in V_{\one}$.


\begin{Theorem}{\rm($G^{\sharp}$ as a locally convex Lie group)} \label{thm:1.4} 
Let $(\oline\rho,\cH)$ be a projective 
unitary representation of a connected locally convex Lie group $G$.
Suppose that  $[\psi] \in \bP(\cH)$ is smooth. Then the following assertions hold: 
\begin{itemize}
\item[\rm(i)] The subset $U_\psi := \{ g \in G\: \oline\rho(g)[\psi]  \not\perp  [\psi]\}$ 
is an open symmetric neighbourhood of $\one$ in $G$. For $g \in U_\psi$, let 
$\rho_\psi(g) \in \U(\cH)$  
be the unique lift of $\oline\rho(g)$ with $\la \psi,\rho_\psi(g) \psi\ra > 0$. 
Then the map $U_\psi \to \cH$ defined by $g \mapsto \rho_\psi(g)\psi$ is smooth and 
$\rho_\psi(g^{-1}) = \rho_\psi(g)^{-1}$ for $g \in U_\psi$. 
\item[\rm(ii)] The central $\T$-extension $G^{\sharp}$
of $G$ carries a natural Lie group structure for which the map
$\tau_{\psi} \colon U_\psi \times \T \to G^\sharp$ defined by $(g,z) \mapsto (g, \rho_\psi(g)z)$ is a 
$\T$-equivariant diffeomorphism onto 
an open subset of~$G^\sharp$. This turns 
$G^\sharp$ into a locally convex principal $\T$-bundle over $G$. 
\item[\rm(iii)] The vector $\psi$ has a smooth orbit map under the unitary 
representation $(\widehat{\rho},\cH)$ of $G^{\sharp}$.
\item[\rm(iv)] If $\phi \in \cH$ is such that $\la \psi,\phi \ra \not=0$ and 
$[\phi] \in \bP(\cH)^\infty$, then $\phi$ is a smooth vector for $\hat\rho$, and 
the Lie group structures on $G^\sharp$ obtained from $[\phi]$ and $[\psi]$ coincide. 
\item[\rm(v)] If $\bP(\cH)^\infty$ is dense in $\bP(\cH)$,
then the Lie group structure on $G^{\sharp}$ does not depend on the 
choice of $[\psi] \in \bP(\cH)^{\infty}$. Moreover, we have $\bP(\cH)^{\infty} = \bP(\cH^{\infty})$,
that is,
a vector $\phi \in \cH\backslash\{0\}$ is smooth for $\widehat{\rho}$ if and only if its 
ray $[\phi]$ is smooth for $\ol\rho$.
\end{itemize}
\end{Theorem}

\begin{proof} (i): We may assume that $\|\psi\| = 1$. 
On the open subset 
\[ V_\psi 
:= \{ [\phi] \in \bP(\cH) \: |\la \psi, \phi\ra| \not= 0\}  
= \{ [\psi + u]\: u \in \psi^\bot \} \] 
which is diffeomorphic to $\psi^\bot$, we have a unique smooth map 
$\sigma_\psi \: V_{\psi} \to \cH$ with 
\begin{equation}  \label{eq:sigmasect}
\|\sigma_\psi([\phi])\| = 1 \quad \mbox{ and } \quad 
\la \psi, \sigma_\psi([\phi]) \ra > 0 \quad \mbox{ for }\quad [\phi] \in V_{\psi}.
\end{equation}
It satisfies 
$\sigma_\psi([\psi + u]) = \frac{\psi+u}{\|\psi+u\|}$ for $u \in \psi^\bot$. 
Since $[\psi]$ has a continuous orbit map, $U_\psi$ 
is an open $\1$-neighbourhood in $G$. For $g \in U_\psi$, there exists a unique 
$\rho_\psi(g) \in \U(\cH)$ with $\oline{\rho_\psi(g)} = \oline\rho(g)$ and 
$\la \psi,\rho_\psi(g) \psi\ra > 0$. Then $\rho_{\psi}(g) \psi = \sigma_\psi(\oline\rho(g)[\psi])$ 
shows that the map $U_\psi \to \cH$ defined by $g \mapsto \rho_\psi(g)\psi$ is smooth. 
For $g \in U_\psi$, the relation 
$\la \psi, \rho_\psi(g)^{-1} \psi \ra =  \la \rho_\psi(g)\psi, \psi \ra > 0$ 
implies that $g^{-1} \in U_\psi$ with $\rho_\psi(g^{-1}) = \rho_\psi(g)^{-1}$. 

(ii) and (iii): It follows from (i) that, in particular, the map 
$G \times G \to \C$, defined by
\[ (g,h) \mapsto \la \psi, \rho_\psi(g) \rho_\psi(h)\psi\ra 
= \la \rho_\psi(g)^{-1} \psi, \rho_\psi(h)\psi \ra 
= \la \rho_\psi(g^{-1})\psi, \rho_\psi(h)\psi \ra \] 
is smooth in a neighbourhood of $(\1,\1)$, so that (ii) and (iii) follow 
from \cite[Thm.~A.4]{Ne10a}. 

(iv) Let $U \subeq U_\psi \subeq G$ be an open symmetric $0$-neighbourhood 
such that $\oline\rho(g)[\phi] \in V_\psi$ for all $g\in U$. 
We have to show that the map 
\[ U \to \cH, \quad g \mapsto \hat\rho(g, \rho_\psi(g))\phi = \rho_\psi(g)\phi \] 
is smooth. Define $\alpha(g) \in \T$ by 
\[ \rho_\psi(g)\phi = \alpha(g) \sigma_\psi(\oline\rho(g)[\phi]).\] 
In view of the smoothness of $\sigma_\psi$, it suffices to show that 
$\alpha$ is smooth. As $\la \psi, \sigma_\psi(\oline\rho(g)[\phi]) \ra > 0$ and 
$|\alpha(g)| = 1$, we have 
\[ \alpha(g) = \frac{\la \psi, \rho_\psi(g)\phi \ra}{|\la \psi, \rho_\psi(g)\phi \ra|}.\] 
Further, the relation $\rho_\psi(g)^{-1}= \rho_\psi(g^{-1})$ shows that 
the function 
\[  U \to \C,  \quad g \mapsto \la \psi, \rho_\psi(g)\phi \ra = \la \rho_\psi(g^{-1})\psi, \phi \ra \] 
is smooth. This proves that $\phi$ is a smooth vector for $\hat\rho$. 

The Lie group structure on $G^\sharp$ constructed from $[\psi]$ has the property that the map 
\[ \tau_{\psi} \colon U_\psi \times \T \to G^\sharp, \quad (g,z) \mapsto (g, \rho_\psi(g)z)\]
is a local diffeomorphism. Let $\rho_\psi$ and $\rho_\phi$ denote lifts of $\oline\rho$ in an open 
symmetric identity neighbourhood $U \subeq U_\phi$ such that 
$\la \psi, \rho_\psi(g)\psi\ra > 0$ and $\la \phi, \rho_\phi(g)\phi\ra > 0$. 
Write $\rho_\psi(g) = \beta_{\psi,\phi}(g)\rho_\phi(g)$ with $\beta_{\psi,\phi}(g)\in \T$. We have to show that 
$\beta_{\phi,\psi}$ is smooth. We have seen above that the function 
\[ g \mapsto \la  \phi, \rho_\psi(g)\phi \ra = \beta_{\psi,\phi}(g) \la \phi, \rho_\phi(g)\phi\ra \] 
is smooth. Since $\la \phi, \rho_\phi(g)\phi\ra = |\la \phi, \rho_\psi(g)\phi\ra|$, it follows that 
\[ \beta_{\psi,\phi}(g) = \frac{ \la \phi, \rho_\psi(g)\phi\ra}{|\la \phi, \rho_\psi(g)\phi\ra|} \] 
is also smooth. This completes the proof of (iv). 

(v) If $[\phi] \in \bP(\cH)^\infty$, then the density of $\bP(\cH)^\infty$ implies the existence 
of an element 
$[\chi]\in \bP(\cH)^\infty$ with $\la \psi,\chi \ra \not= 0 \not= \la \phi,\chi\ra$. 
Then (iv) implies that the 
Lie group structure on $G^\sharp$ does not depend on the choice of $[\psi]$ and that 
$\phi$ is a smooth vector for $\hat\rho$. 
\end{proof}

\begin{Remark} (Equivariance of lifts) Suppose that, in the context of the preceding theorem, we fix a smooth 
ray $[\psi] \in \bP(\cH)$, $g \in G$ and a lift $\rho(g)$ of $\oline\rho(g) \in \PU(\cH)$. 
Then 
\begin{equation}
  \label{eq:ctrafo}
 U_{\rho(g)\psi} = g U_\psi g^{-1} 
\quad \mbox{ and } \quad 
\rho_{\rho(g)\psi}(ghg^{-1}) = \rho(g)\rho_\psi(h)\rho(g)^{-1} \,,
\end{equation}
because 
$\la \rho(g)\psi, (\rho(g)\rho_\psi(h) \rho(g)^{-1}) \rho(g)\psi\ra > 0$. 

\end{Remark}

The following corollary allows us to consider projective smooth unitary representations 
as \emph{linear} 
representations of a central extension, even if the group is not connected.

\begin{Corollary}{\rm(Projective $G$-rep's $\leftrightarrow$ linear $G^{\sharp}$-rep's)} 
\label{lineariseer}
Let $(\overline{\rho},\cH)$ be a smooth projective unitary representation of a 
locally convex Lie group $G$. Then there exists a central Lie group extension 
\[\T \rightarrow G^{\sharp} \rightarrow G\] 
and a smooth unitary representation 
$\hat\rho$ of $G^{\sharp}$ 
such that $\widehat{\rho}(z) = z\mathbf{1}$ on $\T$ 
and $\oline \rho$ is the
corresponding projective unitary representation. 
Moreover, $(\ol\rho',\cH')$ is unitarily equivalent to $(\overline{\rho},\cH)$
if and only if there exists a unitary transformation $U \colon \cH \rightarrow \cH'$
and an isomorphism $\Phi \colon G^{\sharp} \rightarrow G^{\sharp}{}'$ of central Lie group extensions
such that $U \circ \widehat{\rho}(g) = \widehat{\rho}'(\Phi(g)) \circ U$ for 
$g \in G^\sharp$.
\end{Corollary}

\begin{proof} Let $G_0 \subeq G$ be the identity component of $G$. 
We endow $G^\sharp_0$ with the Lie group structure from Theorem~\ref{thm:1.4}. 
Note that $[\psi] \in \bP(\cH)$ is smooth for $\ol\rho$ if and only if it is smooth for 
$\ol\rho|_{G_0}$, so that 
Theorem~\ref{thm:1.4}(v) implies 
that $\bP(\cH)^{\infty} = \bP(\cH^\infty)$, where $\cH^{\infty}$ 
is the space of smooth vectors
for the representation $\hat\rho$ of~$G^\sharp_0$. 

It remains to show that the Lie group structure on the normal subgroup $G^\sharp_0$ of $G^\sharp$ 
extends to a Lie group structure on $G^\sharp$. We have to show that, 
for $(h,W) \in G^\sharp$, the corresponding conjugation automorphism 
$c_{(h,W)}$ of $G^{\sharp}_0$ is smooth in an identity neighbourhood. The manifold structure on 
$G^{\sharp}_0$ is obtained from product maps 
\[ U \times \T \to G^{\sharp}_0, \quad (g,z) \mapsto (g, \rho_\psi(g)z),\]
where $[\psi] \in \bP(\cH)^\infty$ and 
$\rho_\psi(g) \in \U(\cH)$ is specified by $\la \psi, \rho_\psi(g)\psi\ra > 0$. 
In view of \[ (h,W)(g, \rho_\psi(g)z)(h,W)^{-1} = (hgh^{-1}, W \rho_\psi(g) W^{-1} z),\] 
the smoothness of $c_{(h,W)}$ in an identity neighbourhood is equivalent to the 
smoothness of the map 
\[ g \mapsto \beta(g) := \rho_\psi(hgh^{-1})^{-1} W \rho_\psi(g) W^{-1} \in \T \] 
in a $\1$-neighbourhood of $G$. 
For $\phi := W\psi$ we have $[\phi] = \oline\rho(h)[\psi] \in \bP(\cH)^\infty$. 
Further, $\rho_\phi(g) := W \rho_\psi(g) W^{-1}$ satisfies $\oline{\rho_\phi(g)} = \oline\rho(hgh^{-1})$ 
and $\la \rho_\phi(g)\phi,\phi\ra = \la \rho_\psi(g)\psi,\psi\ra > 0$. Therefore the smoothness of $\beta$ 
in an identity neighbourhood follows as in the proof of Theorem~\ref{thm:1.4}(v).
It is clear that the space of smooth vectors for $\widehat{\rho}$ and $\widehat{\rho}|_{G_0^{\sharp}}$
is the same.

To prove the final statement, note that
if $U \colon \cH \rightarrow \cH'$ is a unitary intertwiner between $(\ol\rho,\cH)$
and $(\ol\rho',\cH')$, then 
\[\Phi \colon G^{\sharp} \rightarrow G^{\sharp}{}', \quad (g,V) \mapsto (g,UVU^{-1})\] 
is a continuous
isomorphism of central extensions satisfying $\widehat{\rho}(\Phi(g)) \circ U 
= {U \circ \widehat{\rho}(g)}$.
Since the lifts $\rho'_{U\phi}$ and $\rho_{\psi}$ from Theorem~\ref{thm:1.4}(i) 
satisfy $\rho'_{U\psi}(g) = U\rho_{\psi}(g)U^{-1}$, 
we obtain $\tau'_{U\psi} = \Phi \circ \tau_{\psi}$ for
the local trivialisations $\tau_{\psi}$
and $\tau'_{U\psi}$ of 
Theorem~\ref{thm:1.4}(ii). It follows that $\Phi$ is smooth.
\end{proof}

\begin{Remark}(Identical cocycles) \label{zelfde}
If $U \colon \cH \rightarrow \cH'$ is an isometric intertwiner from $(\ol\rho,\cH)$ to $(\ol\rho',\cH')$,
then the local 2-cocycles $f_{\psi}$ and $f_{U\psi}$ derived from the local lifts $\rho_{\psi}$ and $\rho'_{U\psi}$ 
(cf.\ Definition~\ref{def:1.3})
are not only cohomologous, but 
actually \emph{identical}.
\end{Remark}

In order to classify smooth projective unitary representations of $G$, one first
classifies the central $\T$-extensions $\widehat{G} \rightarrow G$ up to isomorphism, and then
determines the smooth unitary representations of each $\widehat{G}$ (with the property that $\rho(z) = z\one$ for $z\in \T$)
up to unitary equivalence. One then obtains all smooth projective unitary representations,
but there is a slight redundancy in this description.
Indeed,
unitary $\widehat{G}$-representations that differ by a character $\chi \colon G \rightarrow \T$ 
clearly give rise to the same projective $G$-representation, but need not be unitarily equivalent 
as $\widehat{G}$-representations. For connected Lie groups $G$ with a perfect Lie algebra $\fg$, 
this redundancy vanishes because all characters are trivial.

\begin{Proposition}\label{autoriseer}{\rm (Linear representations modulo characters)}
Let $(\widehat{\rho}, \cH)$ and $(\widehat{\rho}',\cH')$ 
be smooth unitary representations of the same
central $\T$-extension $\widehat{G}$, 
with $\widehat{\rho}(z) = \widehat{\rho}'(z)= z \one$ for $z\in \T$.
Then they give rise to unitarily equivalent projective unitary $G$-representations
if and only if there exists a unitary transformation $U \colon \cH \rightarrow \cH'$
and a smooth character $\chi \colon G \rightarrow \T$ such that
\[\widehat{\rho}'(\widehat{g}) = \chi(g)\cdot U\widehat{\rho}(\widehat{g})U^{-1} 
\quad \mbox{ for } \quad \hat G \ni \hat g \mapsto g \in G.\]
If, moreover, the Lie algebra $\fg$ of $G$ is a topologically perfect, then $\chi$ factors through
a character $\chi_{0} \colon \pi_0(G) \rightarrow \T$.
\end{Proposition}

\begin{proof}
If $\Phi \colon \widehat{G} \rightarrow \widehat{G}$ is an automorphism of a
central $\T$-extension $\widehat{G}\rightarrow G$, then it is
in particular an automorphism of principal $\T$-bundles, hence of the form
$\Phi(\widehat{g}) = \widehat{g}\cdot \chi(\widehat{g})$ for a smooth $\T$-equivariant
map $\widehat{\chi} \colon \widehat{G} \rightarrow \T$.
Moreover, $\widehat{\chi}$ must be a group homomorphism because $\Phi$ is.
Since $\Phi(z) = z$ for $z\in \T$, we have $\widehat{\chi}(z) = \one$, so that
$\widehat{\chi}$ factors through a smooth character 
$\chi \colon G \rightarrow \T$.  
The first part of the proposition now follows from Corollary~\ref{lineariseer}.
Since $\dd\chi \colon \fg \rightarrow \R$ is a continuous Lie algebra homomorphism, 
it vanishes on $[\fg,\fg]$. Thus, if $\fg$ is topologically perfect,
$\chi$ is locally constant, hence factors through a character $\chi \colon \pi_{0}(G)\rightarrow \T$.
\end{proof}

\section{Smoothness of projective representations} \label{AppendixA}

In this section, we provide further background on the
notion of \emph{smoothness} for a projective unitary representation
$(\ol\rho,\cH)$ of a locally convex Lie group $G$.
In Subsection~\ref{A1} we obtain 
effective criteria for the smoothness of a ray $[\psi] \in \bP(\cH)$, and 
in Subsection~\ref{A2} 
we 
determine the structure of the set $\bP(\cH)^{\infty}$
of smooth rays. 

\subsection{Smoothness criteria}\label{A1}

The following useful smoothness criterion for unitary representations is proven in \cite[Thm.~7.2]{Ne10b}.
\begin{Theorem}{\rm(Linear smoothness criterion)}\label{kreet}
A vector $\psi \in \cH$ is smooth for a unitary representation $(\rho,\cH)$ if and only if the 
map $G \rightarrow \C, g \mapsto \la \psi, \rho(g) \psi \ra$ 
is smooth in a neighbourhood of $\one\in G$.
\end{Theorem}
In particular, a unitary representation is smooth if it possesses a cyclic vector 
$\Omega \in \cH$ such that $G\rightarrow \C, g \mapsto \la\Omega,\rho(g)\Omega\ra$
is smooth in a neighbourhood of $\one \in G$.
For projective unitary representations, Theorem~\ref{kreet} generalises as follows.

\begin{Theorem}{\rm(Projective smoothness criterion)} \label{projkreet}
A ray $[\psi] \in \bP(\cH)$ is smooth for a projective unitary representation $(\ol\rho,\cH)$
if and only if the transition probability 
\[G \rightarrow \R, \quad g \mapsto p([\psi]\,; \ol\rho(g)[\psi])\]
is smooth in a neighbourhood of 
$\one \in G$ and if, moreover, the local cocycle $f_{\psi}$ (which is then defined)
is smooth in a neighbourhood of $(\one,\one) \in G\times G$. 
\end{Theorem}

\begin{proof}
Since continuity at $\one$ of either the orbit map or the transition probability 
imply the existence of a $\one$-neighbourhood $U_{\one} \subseteq G$
with $\ol\rho(g)[\psi] \not \perp [\psi]$ for $g \in U_\1$, 
the lift $\rho_{\psi}(g)$
of $\ol\rho(g)$ is defined on $U_{\one}$ by the requirement 
$\la\psi,\rho_{\psi}(g)\psi\ra > 0$. Under either assumption, then, the local cocycle 
$f_{\psi} \colon U_{\one} \times U_{\one} \rightarrow \T$ of Definition~\ref{def:1.3}
is defined.
By equivariance,
a ray $[\psi] \in \bP(\cH)$ is smooth if and only if the orbit map 
$G \rightarrow \bP(\cH), g \mapsto \ol\rho(g)[\psi]$ is smooth in a neighbourhood 
$U_{\one}$ of $\one \in G$, which is the case if and only if 
$U_{\one} \rightarrow \cH,  g \mapsto \rho_{\psi}(g)\psi$ is smooth 
(see the proof of Theorem~\ref{thm:1.4}(i)).
By \cite[Thm.~7.1]{Ne10b}, this is the case if and only if 
$U_{\one} \times U_{\one} \rightarrow \C, (g,h) \mapsto \la\rho_{\psi}(g)\psi,
\rho_{\psi}(h)\psi\ra$
is smooth. If $\|\psi\| = 1$, then 
this equals $f_{\psi}(g^{-1},h)\la\psi, \rho_{\psi}(g^{-1}h)\psi\ra$, which in turn 
can be written $f_{\psi}(g^{-1},h) \sqrt{p([\psi], \ol\rho(g^{-1}h)[\psi])}$. The theorem follows.
\end{proof}
This criterion yields the following effective method to prove smoothness of projective 
unitary representations.
\begin{Corollary}{\rm(Smoothness of projective representations)} 
A projective unitary representation $(\ol\rho,\cH)$ is smooth if it possesses a cyclic vector $\Omega \in \cH$ such that 
$g \mapsto p([\Omega], \ol\rho(g)[\Omega])$ and $f_{\Omega}$ are smooth
in a neighbourhood of $\one \in G$ and $(\one,\one) \in G\times G$ respectively.
\end{Corollary}

\subsection{Structure of the set of smooth rays}\label{A2}

We now study the structure of the set 
$\bP(\cH)^\infty$ of smooth rays. 
Our methods are similar to those 
used in \cite{Ne14} to study the set of continuous rays. 

For the following lemma, recall that the topology of the Hilbert manifold $\bP(\cH)$
is induced by the Fubini--Study metric, which is defined by the equation
$d([\psi],[\phi]) = \arccos |\la \psi,\phi\ra| \in [0,\pi/2]$
for unit vectors $\psi,\phi \in  \cH$.
Geodesics in this metric are of the form 
\[ \gamma(t) = [(\cos t)\psi + (\sin t) \chi] \quad \mbox{ for } \quad \|\psi\| = \|\chi\| = 1, \,\la \psi,\chi\ra = 0.\] 
From that one derives easily that points $[\psi],[\phi]$ with $d([\psi],[\phi]) < \pi/2$ are connected 
by a unique minimal geodesic parametrised by arc-length. 

\begin{Lemma}{\rm(Smoothness of generalised midpoint maps)} \label{lem:5.7} For $[\psi], [\phi] \in\bP(\cH)$ with $\la \psi, \phi \ra \not=0$, let 
$\gamma_{[\psi],[\chi]}(t), 0 \leq t \leq \pi/2$ denote the unique minimal geodesic starting in 
$[\psi]$ and passing through $[\phi]$ for $t_0 = d([\psi],[\phi])$. Then, for every $t \in [0,\pi/2]$, the map 
\[ \{([\psi], [\phi]) \in \bP(\cH)^2 \: \la \psi,\phi \ra \not=0\} \to \bP(\cH), 
 \quad 
([\psi],[\phi]) \mapsto \gamma_{[\psi],[\phi]}(t)\] 
is smooth.
\end{Lemma}

\begin{proof} For two unit vectors $\psi$ and $\phi$ with $\la \psi, \phi \ra > 0$, we write 
\[ \chi_{\psi,\phi} 
:= \frac{\phi - \la \phi,\psi\ra \psi}{\|\phi - \la \phi,\psi \ra \psi\|} 
= \frac{\phi - \la \phi,\psi\ra \psi}{\sqrt{1 - |\la \phi,\psi\ra|^2}}.\] 
Then $\la \psi, \chi_{\psi,\phi} \ra = 0$ and we have 
\[ \gamma_{[\psi],[\phi]}(t) = [(\cos t) \psi + (\sin t) \chi_{\psi,\phi}].\] 
Note that $\chi_{\psi,\phi}$ is the unique element 
in $\C \psi + \C \phi$
orthogonal to $\psi$ with $\|\chi_{\psi,\phi}\| = 1$ and 
$\la \chi_{\psi,\phi},\phi \ra > 0$. 
If we choose another representative $\lambda \psi \in [\psi]$ with $|\lambda| = 1$, then 
we also have to replace $\phi$ by $\lambda \phi$ to ensure positivity of the scalar product, and this 
implies that $[(\cos t) \psi + (\sin t) \chi_{\psi,\phi}]$ does not depend on the choice of $\psi$. 

If we choose a smooth local section $\sigma \: U \to \bS(\cH) = \{ \psi \in \cH \: \|\psi\| = 1\}$ of 
$\bS(\cH) \rightarrow \bP(\cH)$ 
on an open neighbourhood $U$ of $[\psi]$ and $\sigma' \: V \to \bS(\cH)$ 
on an open neighbourhood $V$ of $[\phi]$ such that any two lines $[a] \in U$ and $[b] \in V$ are  
not orthogonal, then 
\[ \sigma''([a],[b]) := \Big(\sigma([a]), 
 \frac{|\la \sigma([a]), \sigma'([b])\ra|}{\la \sigma([a]), \sigma'([b])\ra} 
\sigma'([b])\Big)
=: (\sigma([a]), \chi_{\sigma''([a],[b])})\] 
is a smooth section $U \times V \to \bS(\cH)^2$ whose range consists of pairs $(a,b)$ with 
$\la a, b \ra > 0$. We therefore have 
\[ \gamma_{[a],[b]}(t) = [(\cos t) \sigma([a]) + (\sin t) \chi_{\sigma''([a],[b])}],\]
so that the assertion follows from the smoothness of the map $(\psi,\phi) \mapsto \chi_{\psi,\phi}$. 
\end{proof}

This yields the following description of the set of smooth rays: 

\begin{Theorem}{\rm(Structure of the set of smooth rays)} \label{prop:pc} 
There exists a family $(\cD_j)_{j \in J}$ of not necessarily closed 
mutually orthogonal linear subspaces of $\cH$ such that 
\[  \bP(\cH)^\infty = \bigcup_{j \in J} \bP(\cD_j).\] 
In particular, if $\bP(\cH)^\infty$ is dense in $\bP(\cH)$, then it is of the form 
$\bP(\cD)$ for a dense subspace $\cD \subeq \cH$. 
\end{Theorem}

\begin{proof} (cf.\ \cite[Lemma~5.8]{Ne14}) First we show that, if 
$[\psi_0], \ldots, [\psi_n] \in \bP(\cH)^\infty$ are such that 
$\la \psi_j, \psi_{j+1}\ra \not=0$ for $j =0,\ldots, n-1$, then 
\[[\psi_0,\ldots, \psi_n] := \bP(\Spann\{\psi_0,\ldots, \psi_n\})\]
is contained in $\bP(\cH)^\infty$. 
We start with the case $n = 1$. Let 
\[ \Exp \: T(\bP(\cH)) \to \bP(\cH) \] 
denote the exponential map of the Riemannian symmetric space $\bP(\cH)$. 
For $[\chi] = \Exp(v)$ and  $v \in T_{[\psi]}(\bP(\cH))$ with $\|v\| < \pi/2$, 
the whole geodesic arc 
$[\chi_t] = \Exp(tv)$, $0 \leq t \leq 1$, consists of elements not orthogonal 
to $[\psi]$. In view of Lemma~\ref{lem:5.7}, $[\psi], [\chi] \in \bP(\cH)^\infty$ implies that 
$[\chi_t] = \gamma_{[\psi],[\chi]}(t)\in \bP(\cH)^\infty$ for $0 \leq t \leq 1$ because 
$\oline\rho(g)\gamma_{[\psi],[\chi]} = \gamma_{\oline\rho(g)[\psi], \oline\rho(g)[\chi]}$.

Since $\Exp$ is $\PU(\cH)$-equivariant, we 
conclude that, for the action of $G$ on the tangent bundle 
$T(\bP(\cH))$, the set 
$T(\bP(\cH))^\infty$ of $G$-smooth elements has the 
property that, if $v \in T(\bP(\cH))^\infty$ with $\|v\| < \pi/2$, 
then $[0,1]v \subeq T(\bP(\cH))^\infty$. 
Since $G$ acts on $T(\bP(\cH))$ by bundle automorphisms, 
it also follows that, for each $[\psi] \in \bP(\cH)^\infty$, 
the set $T_{[\psi]}(\bP(\cH))^\infty$ is a complex linear subspace. 
This implies that, for two non-orthogonal elements 
$[\psi], [\chi] \in \bP(\cH)^\infty$, the whole projective 
plane $[\psi,\chi] \subeq \bP(\cH)$ consists of $G$-smooth rays. This completes the proof 
for the case $n =1$. 

We now argue by induction. 
Assume that $n > 1$. Then the induction hypothesis implies that 
$[\psi_0,\ldots, \psi_{n-1}] \subeq \bP(\cH)^\infty$. 
If $[\psi] \in [\psi_0,\ldots, \psi_{n}]$ is different from $[\psi_n]$, then 
$[\psi,\psi_n]$ is a projective line intersecting $[\psi_0,\ldots, \psi_{n-1}]$ in some point 
$[\chi]$. Then $[\psi] \in [\chi,\psi_n] \subeq \bP(\cH)^\infty$ follows from the case $n =1$. 
This shows that $[\psi_0,\ldots, \psi_n] \subeq \bP(\cH)^\infty$. 

The preceding arguments imply that every non-empty 
subset $C \subeq \bP(\cH)^\infty$ which cannot be decomposed 
into two non-empty orthogonal subsets is of the form $\bP(\cD)$ for a linear subspace 
$\cD \subeq \cH$. For two such subsets $C_1$ and $C_2$, the corresponding subspaces 
$\cD_1$ and $\cD_2$ are clearly orthogonal. This completes the proof. 
\end{proof}

\section{Lie algebra extensions and cohomology}
\label{sec:6} 

We have seen in Corollary~\ref{lineariseer} that smooth projective unitary representations of a 
locally convex Lie group $G$ are linear representations of a central extension $G^{\sharp}$.
Since those are determined on $G^{\sharp}_{0}$ by their derived Lie algebra 
representation (Proposition~\ref{LArulez}), it is worth while to take a closer look at
central extensions of locally convex Lie algebras.

\begin{Definition} {\rm(Central Lie algebra extensions)} \label{LAUitbreiding}
A \emph{central extension} of a locally convex Lie algebra $\fg$ by $\R$ is 
an exact sequence
\begin{equation}\label{LARijtje}
0 \rightarrow \R \rightarrow \widehat{\fg} \rightarrow \fg \rightarrow 0
\end{equation}
of locally convex Lie algebras such that the image of $\R$ is central in $\widehat{\fg}$.
An \emph{isomorphism} $\phi \colon \widehat{\fg} \rightarrow \widehat{\fg}'$ of
central extensions is an isomorphism of locally convex Lie algebras 
that induces the identity maps on $\fg$ and $\R$.
\end{Definition}
Needless to say, group extensions 
in the sense of Definition~\ref{GroepUitbreiding} give rise to
Lie algebra extensions in the sense of Definition~\ref{LAUitbreiding}. 
Lie algebra extensions are classified by Lie algebra cohomology.

\begin{Definition} \label{def:cohom} {\rm(Continuous LA cohomology)}
The \emph{continuous Lie algebra
cohomology} $H^n(\fg,\R)$ of a locally convex Lie algebra $\fg$ 
is the cohomology of the complex
$C^{\bullet}(\fg,\R)$, where $C^n(\fg,\R)$ consists of the 
continuous alternating linear
maps $\g^n \rightarrow \R$ 
with differential $\delta^n \colon C^{n}(\fg,\R) \rightarrow C^{n+1}(\fg,\R)$ defined by
\[\delta^{n} \omega(\xi_0,\ldots,\xi_{n}):= \sum_{0\leq i<j\leq n}
(-1)^{i+j} \omega([\xi_i,\xi_j],\xi_1,\ldots,\widehat{\xi}_i, \ldots, 
\widehat{\xi}_j, \ldots, \xi_n)\,.\]
\end{Definition}
Given a 2-cocycle $\omega \: \g^2 \rightarrow \R$, 
we define the Lie algebra $\widehat{\fg}_{\omega}$ by
\[
\widehat{\fg}_{\omega} :=  \R \oplus_{\omega} \g
\]
with the Lie bracket
$[(z,\xi), (z',\xi')] := \big(\omega(\xi,\xi'), [\xi,\xi']\big)$. 
Equipped with the obvious maps
$\R \rightarrow \widehat{\fg}_{\omega} \rightarrow \fg$, it
constitutes a central extension 
of $\fg$ by~$\R$. 

\begin{Proposition}{\rm(Cohomological description of central extensions)}
Central extensions of $\fg$ by $\R$ are classified up to isomorphism by $H^2(\fg,\R)$.
\end{Proposition}

\begin{proof} Let $\R \into \hat\g \sssmapright{q}  \g$ be a central extension of $\g$ by $\R$ as above. The Hahn--Banach Theorem guarantees existence\footnote{For 
extensions derived from projective representations, such
maps can be constructed explicitly.}  of a continuous linear
(but, in general, not homomorphic) map 
$\sigma \colon \widehat{\fg} \rightarrow \R$ of (\ref{LARijtje}) that is the identity on $\R$.
The alternating continuous linear map 
$\sigma \circ [\,\cdot\,,\,\cdot\,] \colon \widehat{\fg}^2  \rightarrow \R$
vanishes on $\R \times \widehat{\fg}$, hence drops 
to a 2-cocycle $\omega_{\sigma} \colon \fg^2 \rightarrow \R$.
The difference between two maps $\sigma$ and $\sigma'$ as above 
is continuous and vanishes on $\R$,
hence defines a 1-cochain $\beta \colon \fg \rightarrow \R$.
Thus $\omega_{\sigma'} - \omega_\sigma = \delta^1\beta$ is exact,
and the cohomology class $[\omega_{\sigma}] \in H^2(\fg,\R)$ does not depend on $\sigma$.
If $\phi \colon \widehat{\fg} \rightarrow \widehat{\fg}'$ is an isomorphism of central 
extensions, then $\sigma' := \phi \circ \sigma$ 
is the identity on $\R$, and thus $\omega_{\sigma'} = \omega_{\sigma}$.
\end{proof}

If $(\ol\rho,\cH)$ is a smooth projective unitary representation of $G$,
then the central Lie group extension $\T \rightarrow G^{\sharp} \rightarrow \fg$
of Corollary~\ref{lineariseer} gives rise to a central Lie algebra extension 
$\R \rightarrow \fg^{\sharp} \rightarrow \fg$.

\begin{Proposition}\label{bolhoofd} {\rm(A map from $\bP(\cH^{\infty})$ to the space of cocycles)}
If $\psi\in \cH^{\infty}$ is a unit vector,
then
the class $[\omega_{\psi}] \in H^2(\fg,\R)$ corresponding to the central extension 
$\fg^{\sharp}$ is represented by the 2-cocycle 
\[\omega_{\psi}(\xi,\xi') = -i\la\psi, [\rho_{\psi *}(\xi), \rho_{\psi *}(\xi')] \psi\ra, \]
where 
$\rho_{\psi *} \colon T_{\one}G \cong \g \rightarrow T_{\one}G^{\sharp}\cong \hat\g$ 
is the differential of the local lift 
$\rho_{\psi}$ of {\rm Theorem~\ref{thm:1.4}}. 
In terms of the group 2-cocycle, it is given by
\begin{equation*}
\omega_{\psi}(\xi,\xi') = 
-i \Big(\frac{\partial^2}{\partial t \partial s}\Big|_{s,t=0} f_{\psi}(\gamma_{\xi}(t), \gamma_{\xi'}(s))
- \frac{\partial^2}{\partial t \partial s}\Big|_{s,t=0} f_{\psi}(\gamma_{\xi'}(t), \gamma_{\xi}(s))\Big),
\end{equation*} 
where $\gamma_{\xi}$ denotes a curve $\R \rightarrow G$ 
with $\gamma_{\xi}(0) = \one$ and $\gamma'_{\xi}(0) = \xi$.
\end{Proposition}

\begin{proof} 
Note that $\rho_{\psi *}$ is a continuous linear splitting of the exact sequence
$\R \rightarrow \fg^{\sharp} \rightarrow \fg$. In general, it is not homomorphic. 
The corresponding Lie algebra 2-cocycle 
\begin{equation}\label{reptococycle}
\omega_{\psi}(\xi,\xi') := i(\rho_{\psi}{}_{*}([\xi,\xi']) - [\rho_{\psi}{}_{*}(\xi),\rho_{\psi}{}_{*}(\xi')])
\end{equation}
measures the failure of $\rho_{\psi}{}_{*}$ to be a Lie algebra homomorphism.
Since the real valued function $g \mapsto \la\psi, \rho_{\psi}(g) \psi\ra$ is maximal at $g=\one$,
we have $\la\psi, \rho_{\psi*}(\xi) \psi\ra = 0$ for all $\xi \in \fg$, so the formula for $\omega_{\psi}$
in terms of $\rho_{\psi *}$ follows.
To obtain the formula for $\omega_{\psi}$ in terms of $f_{\psi}$, one differentiates 
Equation~(\ref{eq:cocycl}).
\end{proof}


By Corollary~\ref{lineariseer}, every smooth projective unitary representation of a locally convex Lie group $G$
gives rise to a class $[G^{\sharp}] \in \Ext(G,\T)$,
and to a smooth unitary representation of $G^{\sharp}$.
The Lie group extension $\T \rightarrow G^{\sharp} \rightarrow G$
in turn gives rise to a Lie algebra extension $\R \rightarrow \widehat{\fg} \rightarrow \fg$,
hence to a class $[\omega] \in H^2(\fg,\R)$. 

The converse direction, however, is quite nontrivial;
in general, not every class in $H^2(\fg,\R)$ will integrate to a class in $\Ext(G,\T)$.
The following is an integrability criterion for Lie algebra 2-cocycles in terms of discreteness conditions 
on the \emph{period homomorphism} $\mathrm{per}_{\omega} \colon \pi_2(G) \rightarrow \R$,
which is the extension of
the map $C^{\infty}_{*}(\mathbb{S}^2,G) \rightarrow \R$ 
defined by $\sigma \mapsto \int_{\sigma}\Omega$, where 
$\Omega$ the left invariant 2-form on $G$ with $\Omega_{\one} = \omega$
(see \cite[Def.~5.8]{Ne02}).

\begin{Theorem}{\rm (Integration of Lie algebra cocycles)}
Let $G$ be a connected simply connected Lie group
modelled on a locally convex Lie algebra $\fg$, and let $[\omega] \in H^2(\fg,\R)$. 
We fix the isomorphism $\R \simeq \mathrm{Lie}(\T)$ obtained by 
the exponential function $\exp_\T(t) = e^{2\pi it}$. 
Then the Lie algebra extension $\R \rightarrow \widehat{\fg} \rightarrow \fg$ 
defined by $\omega$
integrates to an extension $\T \rightarrow \widehat{G}\rightarrow G$ of Lie groups 
if and only if $\mathrm{per}_{\omega}(\pi_2(G))\subseteq \Z$.
\end{Theorem}
\begin{proof}
This is \cite[Thm.~7.9]{Ne02}.
\end{proof}


As a byproduct of the construction of the smooth structure on $G^{\sharp}$ (cf. 
Theorem~\ref{thm:1.4}),
we obtain the following necessary condition for a class $[\omega]$ to come from 
a projective unitary representation.

\begin{Proposition}{\rm (Necessary condition for unitarity of cocycles)} 
\label{polarisability}
A necessary condition for a class $[\omega] \in H^2(\fg,\R)$ to correspond to a projective unitary representation
is that for some representative $\omega_{\psi} \in [\omega]$, there exists a continuous positive semi-definite 
sesquilinear form $H_{\psi} \colon \fg_{\C} \times \fg_{\C} \rightarrow \C$ such that 
$\omega_{\psi} = -2\mathrm{Im}(H_{\psi})$ on 
$\fg$. It satisfies $\mathrm{Ker}(H_{\psi}) \subseteq \mathrm{Ker}(\omega_{\psi})_{\C}$.
\end{Proposition}
In particular, $\xi,\eta \in \fg$ satisfy the \emph{Heisenberg uncertainty relations}
\begin{equation}
  \label{eq:uncert}
{\textstyle \frac{1}{2}}|\omega_{\psi}(\xi,\eta)| \leq \|\xi\|_{H_{\psi}} \|\eta\|_{H_{\psi}}\,.
\end{equation}

\begin{proof} 
If $(\ol{\rho},\cH)$ is a smooth projective unitary representation of $G$ 
and $\psi \in \cH^{\infty}$ is a smooth vector, then recall from 
Proposition~\ref{bolhoofd} the formula 
\begin{equation} 
\omega_{\psi}(\xi,\eta) 
= -i \la\psi, [\rho_{\psi*}(\xi), \rho_{\psi*}(\eta)] \psi \ra\,.\label{cocvanpsi}
\end{equation} 
If we define the positive semi-definite sesquilinear form $H_{\psi} \colon \fg_{\C} \times \fg_{\C} \rightarrow \C$ by 
\begin{equation}
H_{\psi}(\xi,\eta) := \la \rho_{\psi *}(\xi) \psi, \rho_{\psi *}(\eta) \psi \ra\,,\label{inpvanpsi}
\end{equation}
then $\omega_{\psi} = -2 \mathrm{Im} H_{\psi}$ on $\fg$.
Note that the kernel of $H$ is precisely the complexification of the stabiliser of $[\psi]$ in $\fg_{\C}$,
that is,
$\mathrm{Ker}(H_{\psi}) = \mathfrak{stab}_{\fg_{\C}}([\psi]) \subseteq \mathrm{Ker}(\omega)$.
\end{proof}

The representative $\omega_{\psi}$ and its corresponding semi-definite sesquilinear form $H_{\psi}$
depend on the choice of a ray $[\psi] \in \bP(\cH^{\infty})$. If we choose a different point 
$\ol{\rho}(g)[\psi]$
in the same $G$-orbit $G[\psi] \subseteq \bP(\cH^{\infty})$, then the corresponding 
$\omega_{\widehat{\rho}(g)\psi}$ and $H_{\widehat{\rho}(g)\psi}$ are related to  
$\omega_{\psi}$ and $H_{\psi}$ by the adjoint representation: 

\begin{Proposition}{\rm (Covariance $\omega$ and $H$)}\label{CocycleCovariance}
If $[\psi'] = \ol{\rho}(g)[\psi]$ is in the $G$-orbit of $[\psi] \in \bP(\cH^{\infty})$, then 
\begin{eqnarray}
\omega_{\widehat{\rho}(g)\psi}(\xi,\eta) &=& \omega_{\psi}(\mathrm{Ad}_{g^{-1}}(\xi), \mathrm{Ad}_{g^{-1}}(\eta))\,,
\label{2vorm}\\
H_{\widehat{\rho}(g)\psi}(\xi,\eta) &=& H_{\psi}(\mathrm{Ad}_{g^{-1}}(\xi), \mathrm{Ad}_{g^{-1}}(\eta))\,.
\label{Kaehler}
\end{eqnarray}
In particular, $H_{\psi}$ and $\omega_{\psi}$ are invariant under the stabiliser $G_{[\psi]} \subseteq G$ 
of~$[\psi]$.
\end{Proposition}

\begin{proof} From \eqref{eq:ctrafo}, we know that 
$\rho_{\hat\rho(g)\psi}(ghg^{-1}) = \hat\rho(g)\rho_\psi(h)\hat\rho(g)^{-1}$, 
which leads to 
\[ \rho_{\hat\rho(g)\psi *}(\xi) = \hat\rho(g)\rho_{\psi *}(\Ad_g^{-1}\xi)\hat\rho(g)^{-1}\] 
and further to 
\[ \rho_{\hat\rho(g)\psi *}(\xi) \hat\rho(g)\psi 
= \hat\rho(g)\rho_{\psi *}(\Ad_g^{-1}\xi)\psi.\] 
The assertion now follows from (\ref{cocvanpsi}) and (\ref{inpvanpsi}).
\end{proof}

\begin{Remark}
The geometric intuition behind Propositions \ref{polarisability} and \ref{CocycleCovariance}
is the following. If the orbit $\mathcal{O} = G[\psi] \subseteq \bP(\cH^{\infty})$ 
happens to be a smooth homogeneous manifold
modelled on $\fg/\fg_{[\psi]}$ (for example if $G$ is finite dimensional), 
then $P = G^{\sharp} \psi \subseteq \cH^{\infty}$ is a smooth principal $\T$-bundle
over $\mathcal{O}$, modelled on $\fg^{\sharp}/\fg_{\psi}$.
The maps $[\psi] \mapsto \omega_{\psi}$ and $[\psi] \mapsto \mathrm{Re}(H_{\psi})$ then
define an equivariant smooth $2$-form and a hermitian metric on $\mathcal{O}$, 
whereas the splitting $s_{\psi} \colon T_{[\psi]}\mathcal{O} \rightarrow T_{\psi}P$
derived from $\rho_{\psi*} \colon \fg/\fg_{[\psi]} \rightarrow \fg^{\sharp}/\fg_{\psi}$ 
yields an equivariant connection 1-form on 
$P$ with curvature $-\omega$.  
This makes $P \rightarrow \mathcal{O}$ into a prequantum $\T$-bundle with equivariant Riemannian metric 
$\mathrm{Re}(H)$.
\end{Remark}

\section{The main theorem}\label{sec:7}

We have seen that for locally convex Lie groups, there is a correspondence 
between smooth projective unitary representations of $G$
and smooth linear unitary representations of a central extension $G^{\sharp}\rightarrow G$.
Moreover, for connected simply connected Lie groups $G$ modelled on a barrelled Lie  
algebra $\fg$, there is a correspondence between smooth unitary representations of $G^{\sharp}$
and regular unitary representations of its Lie algebra $\fg^{\sharp}$.
We now give a formulation of this correspondence which carefully takes the intertwiners into account.  

\begin{Definition} {\rm(Representation categories)}
Let $G$ be a Lie group modelled on a locally convex Lie algebra~$\fg$. 
Define the categories {\it I}, {\it II}, {\it III}, and {\it III}$_{r}$
as follows:
\begin{enumerate}
\item[{\it I}.] \label{item1}
The category of triples $(\ol\rho,\cH,[\psi])$, where $(\ol\rho,\cH)$ is a nonzero smooth projective unitary 
representation of $G$ and 
$[\psi] \in \bP(\cH)^{\infty}$ is a distinguished smooth ray.
A morphism $(\ol\rho,\cH,[\psi]) \rightarrow (\ol\rho',\cH',[\psi'])$ is a linear isometry
$U \colon \cH \rightarrow \cH'$ with $[U\psi] = [\psi']$ and 
$\ol{U} \circ \ol\rho = \ol \rho' \circ \ol{U}$.
\item[{\it II}.] \label{item2}
The category of quadruples $(\widehat{G}\rightarrow G, \rho,\cH, [\psi])$, where 
$\widehat{G}\rightarrow G$ is a smooth central $\T$-extension of $G$, $(\rho,\cH)$ is a smooth 
unitary representation of $\widehat{G}$ such that $\rho(z) = z \one$ on $\T$, 
$[\psi] \in \bP(\cH^{\infty})$, and the morphisms are pairs $(\Phi,U)$, where 
$\Phi \colon \widehat{G} \rightarrow \widehat{G}{'}$ is an isomorphism of central extensions 
and $U \colon \cH \rightarrow \cH'$ is a linear isometry satisfying $U\rho(g) = \rho'(\Phi(g))U$
and ${[U\psi]=[\psi']}$. 
\item[{\it III}.] \label{item3} 
The category of quadruples $(\widehat{\fg}\rightarrow \fg, \pi, V, [\psi])$, where 
$\widehat{\fg}\rightarrow \fg$ is a central $\R$-extension of locally convex Lie algebras,
$(\pi,V)$ is a continuous 
unitary representation of $\widehat{\fg}$ such that $\pi(1) = \one$ 
(cf.~Definition~\ref{LAcontinuity}), $[\psi] \in \bP(V)$, 
and morphisms are pairs $(\phi,U)$ where 
$\phi \colon \widehat{\fg} \rightarrow \widehat{\fg}{}{'}$ is an 
isomorphism of central extensions 
and $U \colon V \rightarrow V'$ is a linear isometry satisfying $U\pi(\xi) = \pi'(\phi(\xi))U$
and $[U\psi] = [\psi']$.
\item[{\it III}$_{\mathrm{r}}$.] The full subcategory of {\it III} where 
the image $\mathrm{per}_{\omega}(\pi_2(G)) \rightarrow \R$ of the 
period homomorphism for the class $[\omega] \in H^2(\fg,\R)$ of 
$\widehat{\fg}\cong \R \oplus_\omega \g$ 
is contained in $\Z$, and $(\pi,V)$ is (strongly) regular.
\end{enumerate}
The case $\cH = \{0\}$ is excluded in all categories.
\end{Definition}

\begin{Remark} (Morphisms are intertwiners for perfect, connected $G$)
If $G$ is connected and $\fg$ is topologically perfect, then the isomorphisms
$\Phi \colon \widehat{G} \rightarrow \widehat{G}{}'$ and 
$\phi \colon \widehat{\fg} \rightarrow \widehat{\fg}{}'$ are unique 
by Proposition~\ref{autoriseer}. In this case, isomorphic central extensions 
can be canonically identified, and the morphisms are simply isometric intertwiners $U$ 
with $[U\psi] = [\psi']$.
\end{Remark}

\begin{Theorem}{\rm (Main theorem)}\label{maintheorem}
Let $G$ be a Lie group modelled on a locally convex Lie algebra~$\fg$. 
\begin{itemize}
\item[A)] There exists an equivalence of categories between {\it I} and {\it II}.
\item[B)] If $G$ is connected, then the derived representation functor 
$\dd \colon \mathrm{{\it II}} \rightarrow \mathrm{{\it III}}$
 is fully faithful. If $G$ is connected and regular, then $\dd$ lands in 
$\mathrm{{\it III}}_{r}$.
\item[C)] If $G$ is a connected, simply connected, regular Lie group modelled on a barrelled 
Lie algebra $\fg$, then 
there exists an integration functor, denoted $\mathbf{I} \colon \mathrm{{\it III}}_{r} \rightarrow \mathrm{{\it II}}$,
which is left adjoint to the derived representation functor $\dd \colon\mathrm{{\it II}} \rightarrow \mathrm{{\it III}}_{r} $.
\end{itemize}
\end{Theorem}

\begin{proof} 
The functors we need are already constructed at the level of objects, so we need only check 
the above statements and functoriality
at the level of morphisms.
\noindent A) The equivalence of $\mathrm{{\it I}}$ and $\mathrm{{\it II}}$ is essentially 
Corollary \ref{lineariseer}. Considering a $\widehat{G}$-representation as a 
projective $G$-representation yields a functor
$\mathrm{{\it II}}\rightarrow \mathrm{{\it I}}$. We show that the construction of 
$(G^{\sharp}, \rho, \cH, \psi)$ out of $(\ol\rho, \cH, \psi)$ is functorial.
By Remark~\ref{zelfde} and the fact that the construction of the cocycle only depends on the 
class $[\psi]$ of $\psi$, 
the extensions $G^{\sharp}$ and $G^{\sharp}{}'$ derived from 
$(\ol\rho, \cH, \psi)$ and $(\ol\rho', \cH', \psi')$
are \emph{the same} if there exists  
a linear isometric intertwiner $U \colon \cH \rightarrow \cH'$ with 
$[U\psi] =[\psi']$.
This means that we can define the functor to map
$U$ to $(\Phi,U)$ with $\Phi = \mathrm{Id}$,
\[
(G^{\sharp}, \rho, \cH, \psi) \smapright{(\mathrm{Id},U)}(G^{\sharp}, \rho', \cH', \psi') \,.
\]
The composition 
$\mathrm{{\it I}} \rightarrow \mathrm{{\it II}} \rightarrow \mathrm{{\it I}}$ is the identity, and the natural 
transformation between the identity and 
$\mathrm{{\it II}} \rightarrow \mathrm{{\it I}} \rightarrow \mathrm{{\it II}}$ is 
given on $(\widehat{G},\rho, \cH, \psi)$ by the pair $(\Phi, U)$ where 
$U \colon \cH \rightarrow \cH$ is the identity and $\Phi \colon \widehat{G} \rightarrow G^{\sharp}$
is fixed by requiring that $\Phi(\widehat{g})$ cover $g\in G$ and 
$\rho(\widehat{g}) = \rho^{\sharp}(\Phi(\widehat{g}))$.
B) The derived representation functor $\dd$ takes $(\widehat{G}, \rho, \cH, \psi)$ to 
$(\widehat{\fg}, \dd\rho, \cH^{\infty}, \psi)$
and $(\Phi, U)$ to $(\dd\Phi, U|_{\cH^{\infty}})$. 
It is fully faithful for connected $G$ by  Proposition~\ref{LArulez}. 
The image of the period map is discrete for connected $G$
because by \cite[Prop.~5.11]{Ne02}, $\mathrm{per}_{\omega}$ is minus the connecting homomorphism 
$\delta \colon \pi_{2}(G) \rightarrow \Z$
of the principal $\T$-bundle $\widehat{G}\rightarrow \Z$.
If $G$ is regular, then so is $(\dd\rho, \cH^{\infty})$ by Proposition~\ref{derivedisregular}. 
Together, discreteness and regularity imply that the Lie functor lands in 
$\mathrm{{\it III}}_{r}$.
C) We describe the integration functor $\mathrm{{\it III}}_{r} \rightarrow \mathrm{\it II}$.
The ray $[\psi]$ picks out a distinguished splitting $\sigma \colon \fg  \rightarrowtail \widehat{\fg}$
of $\widehat{\fg} \rightarrow \fg$
by the requirement $\langle\psi, \pi(\sigma(\xi)) \psi\rangle = 0$ for $\xi \in \fg$,
and hence a distinguished element $\omega$ in the class $[\omega]$ corresponding to 
$\widehat{\fg} \rightarrow \fg$ 
(this is the infinitesimal version of equation~\eqref{reptococycle}).
If $(\phi,U)$ is a morphism from $(\widehat{\fg}, \pi, V, [\psi])$ to 
$(\widehat{\fg}{'}, \pi', V', [\psi'])$, then the splitting 
$\sigma' \colon \fg \rightarrowtail \widehat{\fg}{}'$ is precisely $\sigma ' = \phi \circ \sigma$,
so that $\omega'$ is \emph{identical} to $\omega$.
Although the construction of group extensions from Lie algebra extensions 
described in \cite[\S~6]{Ne02} is highly non-canonical (it depends on a choice of a system of paths in $G$),
it takes as input only a Lie algebra 2-cocycle $\omega$, so it produces the \emph{same}
central extension $\widehat{G}$ for $\widehat{\fg}$ and $\widehat{\fg}{}'$.
We construct the smooth group representation of $\widehat{G}$ on $\cH_{V}$
by solving the ODE
$\frac{d}{dt}\psi_t = \pi(\delta^{R}g_t)\psi_t$ with initial condition $\psi_0$ 
as outlined in Proposition~\ref{regularisderived}.
(The identification between $\widehat{\fg}$ and the Lie algebra
$\fg \oplus_{\omega}\R$ of $\widehat{G}$ needed to evaluate $\pi(\delta^{R}g_t)$ is fixed by the 
splitting~$\sigma$.)
If $U \colon V \rightarrow V'$ is an isometric intertwiner, then $\psi'_t = U\psi_t$ is a solution to
$\frac{d}{dt}\psi'_t = \pi'(\delta^{R}g_t)\psi'_t$ with initial condition $\psi'_0 = U\psi_0$, so 
its closure $U \colon \cH_{V} \rightarrow \cH_{V'}$ is an intertwiner at the group level.
To check that $\mathbf{I}$ is left adjoint to $\mathbf{\dd}$, we provide a natural isomorphism
\[\mathrm{hom}_{{\it II}}(\mathbf{I} \, - , -) 
\simeq 
\mathrm{hom}_{{\it III}_{r}}(-, \dd\, -).
\]
By the above argument, we may identify the groups and Lie algebras on both sides.
For a unitary $G$-representation $(\pi, \cH)$, 
let $U \colon \cH_{V} \rightarrow \cH$ be an isometric intertwiner in {\it II}. 
Since $V \subseteq \cH_{V}$ is contained in $\cH_{V}^{\infty}$, we have
$UV \subseteq \cH^{\infty}$, and
$U|_{V} \colon V \rightarrow \cH^{\infty}$ is an isometric intertwiner in {\it III}$_{r}$.
Conversely, every isometric intertwiner $U|_{V} \colon V \rightarrow \cH^{\infty}$
extends to the closure $\cH_{V} \rightarrow \cH$, hence defines an isometric intertwiner 
in {\it II}. This shows that $(\mathbf{I} , \dd)$ is an adjoint pair.
\end{proof}

\begin{Remark}{\rm(Adjunction, but not equivalence)}
The adjoint pair $(\mathbf{I}, \dd)$ is in general not an equivalence of categories,
because non-equivalent Lie algebra representations can correspond to equivalent group 
representations. For example, the Lie algebra representations on the
spaces of smooth and analytic vectors for a single unitary group representation are in general 
not unitarily equivalent.
\end{Remark}

\section{Covariant representations}
\label{sec:8}

In this section, we consider a locally convex Lie group $G$, endowed with a smooth 
action of a locally convex Lie group $R$, given by a homomorphism 
${\alpha \: R \to \Aut(G)}$. 
The semidirect product $G \rtimes_{\alpha}R$ is then again a locally convex Lie group
with Lie algebra $\fg \rtimes_{D\alpha} \fr$, the ``crossed product'' 
by $D\alpha \: \fr \to \mathrm{der}(\fg)$.

This setting, especially with $R = \R$, is frequently encountered 
in the representation theory of infinite dimensional Lie groups, 
in particular for loop groups
(cf.\ \cite{PS86} and Section \ref{sec:loopgroups}) and gauge groups
(cf.\ \cite{JN15}).


\begin{Definition}(Covariant unitary representations) A triple $(\rho, U, \cH)$ of a 
Hilbert space $\cH$ with 
unitary representations 
$\rho$ of $G$ and
$U$ of $R$ is called a \emph{covariant unitary representation} if
\[ U_t \rho(g) U_t^{-1} = \rho(\alpha_t g)\quad \mbox{ for } \quad g \in G, t \in R\] 
or, equivalently, if the map $\rho_{U} \colon G\rtimes_{\alpha} R \rightarrow \U(\cH)$
defined by
\[\rho_{U} (g,t) := \rho(g) U_t\] 
is a unitary representation of ${G \rtimes_\alpha R}$. 
We call $(G,U)$ \emph{continuous} or \emph{smooth} if the corresponding representation
$\rho_U$
of ${G \rtimes_{\alpha}R}$
is continuous or smooth, respectively. 
%
\end{Definition}
\begin{Remark}{\rm(Smoothness of covariant representations)}
A covariant representation $(\rho,U,\cH)$ is continuous if and only if both $\rho$ and $U$
are continuous. If $(\rho,U,\cH)$ is smooth, then both $\rho$ and $U$ are smooth.
The converse may not hold because the intersection 
$\cH^\infty(\rho_U) = \cH^\infty(\rho) \cap \cH^\infty(U)$ 
(cf.\ Theorem~\ref{kreet}) 
may not  be dense in $\cH$. 
\end{Remark}

\begin{Definition} (Projective covariant representations) 
A triple $(\oline\rho, \oline U, \cH)$ of a 
Hilbert space $\cH$ with projective unitary representations $\ol\rho$ of $G$
and $\ol U$ of $R$ is called a 
{\it covariant projective unitary representation} if 
\[ \oline{U_t} \oline\rho(g) \oline{U_t}{}^{-1} 
= \oline\rho(\alpha_t g)\quad \mbox{ for } \quad g \in G, t \in R\] 
or, equivalently, if the map $\ol{\rho_{U}} \colon G \rtimes_\alpha R \rightarrow \PU(\cH)$ 
defined by 
\[\ol{\rho_{U}} (g,t) := \oline\rho(g) \oline{U_t}\] is a 
projective unitary representation of the semi\-di\-rect product 
$G \rtimes_\alpha R$. 
We call $(\oline \rho, \ol U, \cH)$ \emph{continuous} or \emph{smooth}
if the corresponding projective unitary representation of 
$G\rtimes_{\alpha} R$ is continuous or smooth, respectively.
\end{Definition}


\begin{Remark}\label{Rk:2.9} {\rm(Thm.~\ref{maintheorem} in covariant context)}
If $(\oline\rho, \ol U, \cH)$ is a smooth covariant projective unitary representation, then applying 
Corollary~\ref{lineariseer} to the projective representation $\ol{\rho_{U}}$ of $G\rtimes_{\alpha}R$,
we obtain a central 
extension 
\[
\T \rightarrow (G\rtimes_{\alpha}R)^{\sharp} \rightarrow G\rtimes_{\alpha}R 
\]
with a smooth unitary representation $\widehat{\rho}$ of $(G\rtimes_{\alpha}R)^{\sharp}$ on $\cH$ that induces
both $\oline\rho$ and $\oline U$.
From Proposition~\ref{LArulez}, we see that the restriction of $\ol{\rho_{U}}$ to $(G \rtimes_{\alpha}R)_{0}$
is determined up to unitary equivalence by the 
derived representation $\dd\widehat{\rho}$
of the central extension $\R \rightarrow (\fg \rtimes_{D\alpha}\fr)^{\sharp} \rightarrow \fg \rtimes_{D\alpha}\fr$.
Theorem~\ref{maintheorem} provides
necessary and sufficient integrability criteria.
\end{Remark}

In view of the importance of the special case $R = \R$,
we now describe the Lie algebra extensions 
for this particular situation in more detail.
If $\alpha$ is a 
smooth $\R$-action on $G$ 
with infinitesimal generator $D \in \mathrm{der}(\fg)$, 
then $G\rtimes_{\alpha} \R$ is a locally convex Lie group
with Lie algebra $\fg \rtimes_{D}\R$. We now have (cf.\ Proposition~\ref{bolhoofd}):

\begin{Proposition}{\rm(The case $R = \R$)}\label{erisereenjarig}
A ray $[\psi] \in \bP(\cH)$ yields a cocycle 
${\omega_{\psi} \: (\g\rtimes_D \R)^2 \to \R}$
and an isomorphism between $(\fg \rtimes_{D}\R)^{\sharp}$
and the Lie al\-ge\-bra
\begin{equation}
\widehat{\fg} := \widehat{(\fg \rtimes_{D} \R)}_{\omega_{\psi}} =  \R \oplus_{\omega_{\psi}} (\g\rtimes_{D}\R)\,,
\end{equation}
with Lie bracket
\begin{equation} 
[(z,x,t), (z',x',t')] = \big(\omega_{\psi}(x,t;x',t'), [x,x'] + t D(x') - t' D(x), 0\big)\,.
\end{equation}
\end{Proposition}

\begin{Remark} (Equivariant sections and cocycles) \label{rem:inv-cocyc} 
Suppose that
$(\oline\rho, \oline U,\cH)$ is a smooth covariant projective unitary representation 
of $G \rtimes_\alpha R$ and that $\oline U$ is induced by the unitary representation 
$U \: R \to \U(\cH)$. Then $R$ acts on the central extension $G^\sharp$ by 
\[ \hat\alpha_t(g,U) = (\alpha_t(g), U_t U U_{t}^*).\] 

Assume that $\psi$ is an eigenvector for $U$ and that $\chi \: R \to \T$ is 
the corresponding character, defined by 
$U_t \psi = \chi(t)\psi$ for $t \in R$. Then 
$[\psi]$ is fixed under the action of $R$ on $\bP(\cH)$, and  
this implies that $[\psi]$ is invariant under every $\oline U_t$. 
We further obtain for 
$\eta \in \psi^\bot$ and the canonical section 
$\sigma_\psi \: V_\psi \to \cH$ from \eqref{eq:sigmasect} 
the following equivariance relation:  
\begin{align*}
\sigma_\psi(\oline{U_t}[\psi + \eta]) 
&= \sigma_\psi([\chi(t)\psi + U_t \eta])
= \sigma_\psi([\psi + \chi(t)^{-1}U_t \eta])
= \frac{\psi + \chi(t)^{-1}U_t \eta}{\|\psi + \chi(t)^{-1}U_t \eta\|}\\
&= \chi(t)^{-1} U_t \frac{\psi + \eta}{\|\psi + \eta\|}
= \chi(t)^{-1} U_t \sigma_\psi([\psi + \eta]).
\end{align*}
This leads with the terminology from Theorem~\ref{thm:1.4} to 
\begin{align*}
 U_t \rho_\psi(g)\psi 
&= U_t \sigma_\psi(\overline{\rho}(g)[\psi]) 
= \chi(t) \sigma_\psi(\oline{U_t}\overline{\rho}(g)[\psi])
= \chi(t) \sigma_\psi(\overline{\rho} (\alpha_t(g))[\psi])\\
&= \chi(t) \rho_\psi(\alpha_t(g))\psi 
= \rho_\psi(\alpha_t(g)) U_t \psi,
\end{align*}
so that 
\[ U_t \rho_\psi(g)= \rho_\psi(\alpha_t(g)) U_t \quad \mbox{ for } \quad g \in U_\psi, t \in R.\] 
For the cocycle $f \: G \times G \to \T$ defined by \eqref{eq:cocycl} 
in Definition~\ref{def:1.3}, this leads to 
\begin{align*}
 f(\alpha_t(g),& \alpha_t(h)) \1
= \rho_\psi(\alpha_t(gh))^{-1} \rho_\psi(\alpha_t(g)) \rho_\psi(\alpha_t(h))\\
&= U_t \rho_\psi(gh)^{-1} \rho_\psi(g) \rho_\psi(h) U_t^{-1}
=  \rho_\psi(gh)^{-1} \rho_\psi(g) \rho_\psi(h) = f(g,h) \1.
\end{align*}
We conclude that the restriction of $f$ to the pairs 
$(g,h) \in U_\psi^2$ with $gh \in U_\psi$ is $R$-invariant. 
It follows in particular that the action of $\R$ on $\T \times U_\psi \subeq G^\sharp \cong \T \times_f G$ 
is given by 
\[ \hat\alpha_t(z,g) := (z, \alpha_t(g)),  \] 
and hence that it is smooth on this subset. 
We conclude that each $\hat\alpha_t$ is a smooth automorphism of $G^\sharp$ 
and that the action $\hat\alpha$ is smooth on the identity component~$G^\sharp_0$ 
of~$G^\sharp$. 
\end{Remark}

\section{Admissible derivations}\label{AppendixB}


In Subsection~\ref{admissible}, we study 2-cocycles on Lie algebras of the 
form $\fg \rtimes_{D} \R$, where $D$ is an admissible derivation on 
a locally convex Lie algebra $\fg$. These arise naturally in the context of the preceding 
section. 
In Subsection~\ref{cocposen}, we give some extra information on cocycles
that come from representations which are either periodic or of positive energy.

\subsection{Admissible derivations}\label{subsec:admissible}

In this section, we will assume that our derivation is of the following type.
\begin{Definition}{\rm (Admissible derivations)}\label{admissible} 
A continuous derivation $D$ of a locally convex Lie algebra $\fg$ is called
admissible if 
$D\fg \subseteq \fg$ is a closed subspace
and if the 
sequence 
\[
0 \rightarrow \mathrm{Ker}(D) \hookrightarrow \fg \stackrel{D}{\longrightarrow}
D\fg \hookrightarrow \fg \stackrel{q}{\longrightarrow}\mathrm{Coker}(D) \rightarrow 0 
\]
admits continuous linear splitting maps at $\mathrm{Coker}(D)$ and $D\fg$. 
In other words, we require that the 
surjections
$D \colon \fg \rightarrow D \fg$ and 
$q \colon \fg \rightarrow \mathrm{Coker}(D)$
admit continuous sections $I \colon D\fg \rightarrow \fg$ 
and $\sigma \colon \mathrm{Coker}(D) \rightarrow \fg$,
so that $\fg$ admits the direct sum decompositions
$\fg \simeq \mathrm{Ker}(D) \oplus ID\fg$ and
$\fg \simeq D\fg \oplus \sigma q \fg$ of locally convex  
vector spaces.
\end{Definition}


\begin{Proposition}{\rm (Periodic implies admissible)} \label{prop:circleaction}
Let $\fg$ be a complete locally convex 
complex Lie algebra with a smooth $1$-periodic action
$\alpha \colon \R \rightarrow \mathrm{Aut}(\fg)$, i.e., $\alpha$ factors through 
an action of $\T \cong \R/\Z$.  
Let 
$D := \frac{d}{dt}|_{t=0}\,\alpha_t$
be the corresponding derivation.
Then $D$ is admissible.
\end{Proposition}

\begin{proof}
If $\fg_{k} := \{\xi \in \fg\,;\, \alpha_{t}(\xi) = e^{2\pi ikt}\xi\}$, then 
$\fg = \widehat{\bigoplus}_{k\in \Z} \fg_{k}$ as the closure of a direct sum 
of locally convex spaces.
This yields a direct sum decomposition $\fg = \mathrm{Ker}(D)\oplus \mathrm{Im}(D)$ 
with $\mathrm{Ker}(D) = \fg_{0}$ and the restriction of $D$ to 
$\mathrm{Im}(D) = \widehat{\bigoplus}_{k\in \Z/\{0\}} \fg_{k}$ is invertible.
\end{proof}

\begin{Definition} \label{def:9.3} {\rm (Gauge algebra)}
Let $K$ be a finite dimensional Lie group with Lie algebra $\fk$ and let 
$P \rightarrow M$ be a principal $K$-bundle.
Then $\Ad(P) := {P\times_{\Ad}\fk}$ is a bundle of Lie algebras with 
typical fibre $\fk$. We define the 
(compactly supported) \emph{gauge algebra of $P$}
to be the locally convex Lie algebra
$\fg := \Gamma_{c}(\Ad(P))$
of compactly supported smooth sections of $\Ad(P)$, equipped with the
pointwise Lie bracket and the canonical locally convex topology.
\end{Definition}

\begin{Remark}
If $P\rightarrow M$ is the trivial bundle, then 
$\Gamma_{c}(\Ad(P)) \simeq C_{c}^{\infty}(M,\fk)$.
\end{Remark}

We will be interested in the situation where
$P\rightarrow M$ carries a $K$-equivariant $\R$-action
$\gamma \colon \R \rightarrow \mathrm{Aut}(P)$ 
with generator $\bv \in \Gamma(TP)^{K}$.
We then obtain a smooth \mbox{$1$-parameter} family of automorphisms 
$\alpha \colon \R \rightarrow \mathrm{Aut}(\fg)$, hence a continuous derivation
$D := \frac{d}{dt}|_{t=0}\,\alpha_{t}$.
Under the isomorphism $\Gamma(\Ad(P)) \simeq C^{\infty}(P,\fk)^{K}$,
it is given by $D\xi = L_{\bv}\xi$.

If the $\R$-action on $P$ factors through a $\mathbb{T}$-action, then we are in the 
setting of Proposition \ref{prop:circleaction}, so that $D$ is admissible.
The motivating example of an admissible derivation, however, is the following 
(cf.\ \cite{JN15}).

\begin{Proposition}\label{propfree} {\rm(Admissible derivations for gauge groups)}
Suppose that the $K$-equivariant $\R$-action on $P\rightarrow M$ is 
proper and free on $M$. Let $\Sigma := M/\R$ denote the corresponding quotient 
manifold.  
Then $P \simeq P_{\Sigma}\times \R$ with $P_{\Sigma} \rightarrow \Sigma$
a principal $K$-bundle over $\Sigma$. 
Moreover, the derivations on $\Gamma_{c}(\Ad(P))$ and $\Gamma(\Ad(P))$
defined by $D\xi(\eta,t) = \frac{d}{dt}\xi(\eta,t)$ are admissible.
\end{Proposition}

\begin{proof}
As the $\R$-action is proper and free, $M$ and $P$ are 
(automatically trivial) 
principal $\R$-bundles over
$\Sigma := M/\R$ and $P_{\Sigma} := P/\R$ respectively, 
hence $P \simeq P_{\Sigma}\times \R$ with $P_{\Sigma} \rightarrow \Sigma$
a principal $K$-bundle. The formula for $D$ is clear. 

First, we consider the case $\fg := \Gamma_{c}(\Ad(P))$.
Let $I \colon \fg \rightarrow \Gamma(\Ad(P))$
be the integration $I\xi(\eta,t) := \int_{-\infty}^{t}\xi(\eta,\tau) d\tau$
and set $I_{\infty}(\xi)(\eta) := \lim_{t\rightarrow \infty}I\xi(\eta,t)$.
Then $I_{\infty} \colon \fg \rightarrow \Gamma_{c}(\Ad(P_{\Sigma}))$
is continuous with $\mathrm{Ker}(I_{\infty}) = D\fg$, hence 
$D\fg$ is closed.
Furthermore, $D$ is injective on $\fg$ and $D \circ I \circ D = D$, 
so $I|_{D\fg}$ is a continuous section of $D\colon \fg \rightarrow D\fg$.
Finally, the map $I_{\infty}$ yields an isomorphism 
$\mathrm{Coker}(D) \rightarrow \Gamma_{c}(\Ad(P_{\Sigma}))$,
so that any bump function $\phi \in C^{\infty}_{c}(\R)$ of integral 1
yields a continuous section $\sigma \colon \mathrm{Coker}(D) \rightarrow \fg$ by
$\sigma([\xi])(\eta,t) := \phi(t)I_{\infty}(\xi)(\eta)$.

Secondly, we consider the case $\fg := \Gamma(\Ad(P))$. 
Then $D \colon \fg \rightarrow \fg$ is surjective and the section 
$I\colon D\fg \rightarrow \fg$ is given by 
$I\xi(\eta,t):= \int_{0}^{t}\xi(\eta,\tau)d\tau$.
\end{proof}

The requirement that the $\R$-action be either periodic 
or proper and free cannot be dispensed with, as is shown by the following:
\begin{Example} {\rm (Non-admissible derivation)}
Let $\fg := C^{\infty}(T^2,\fk)$ with $T^2 := \R^2/\Z^2$ the $2$-torus with 
{$1$-periodic} coordinates $\theta_1, \theta_2$, and let $\fg_{\C}$
be its complexification. 
Let $D\colon \fg_{\C} \rightarrow \fg_{\C}$ be the derivation 
$D := \partial_{\theta_1} + \tau \partial_{\theta_2}$
with $\tau \in \R \setminus \Q$. Then
$D\fg$ is not closed; choose (infinite) sequences $(n_{k,1})_{k\in \N}$
and $(n_{k,2})_{k\in \N}$ of integers with $|n_{k,1} + \tau n_{k,2}| < 2^{-2k}$.
For $X\in \fk_{\C}$, we set 
\[X_{n_{k,1},n_{k,2}} := \exp\left(i (n_{k,1}\theta_1 + n_{k,2}\theta_{2})\right)\,.\]
Then the Fourier series
$\sum_{k=0}^{\infty}D\left(2^k
X_{n_{k,1},n_{k,2}}\right)$ 
defines a smooth $\fk_{\C}$-valued function on the torus which is in $\ol{D\fg}$,
because all the partial sums are in $D\fg$, 
but since 
$\sum_{k=0}^{\infty}2^k
X_{n_{k,1},n_{k,2}}$ does not converge, it is not in $D\fg$ itself. 
\end{Example}

\begin{Definition} {\rm(Translation invariant cohomology)} Recall from Definition~\ref{def:cohom} 
the continuous Chevalley--Eilenberg complex $C^{\bullet}(\fg,\R)$. 
We denote by $H_{D}^{\bullet}(\fg,\R)$ the cohomology of the subcomplex
$C^{\bullet}(\fg,\R)^{D}$ of cochains 
$\omega$ annihilated by $D$ under the action 
\[ 
D\omega(\xi_1,\ldots,\xi_n) := \sum_{i=1}^{n}\omega(\xi_1,\ldots,D\xi_i,\ldots, \xi_n).\]
\end{Definition}

\begin{Proposition}{\rm ($H^2(\fg\rtimes_{D}\R,\R)$ vs.\ $H^2_{D}(\fg,\R)$)} \label{DInvariant}
If $D$ is an admissible derivation of a locally convex Lie algebra $\fg$, then 
we have the following exact sequence:
\begin{align*}
0 \rightarrow 
\left((D\fg \cap \ol{[\fg,\fg]})/D\ol{[\fg,\fg]})\right)'
\sssmapright{\alpha} H_{D}^2(\fg,\R) 
&\sssmapright{\beta} H^2(\fg\rtimes_{D} \R, \R) \\
&\sssmapright{\gamma} H^1(\mathrm{Ker}(D),\R). 
\end{align*}
Here $\alpha(\lambda)$ is the class of the 
$D$-invariant $2$-cocycle $\delta(\lambda)$,
$\beta([\omega])$ is the class of the cocycle 
$\tilde\omega((x,t), (y,s)) := \omega(x,y)$, and 
$\gamma([\tilde\omega]) = (i_D\tilde\omega)\res_{\ker D}$. 
In particular, $H^2(\fg\rtimes_{D}\R,\R) \simeq H^2_{D}(\fg,\R)$ if both
$\fg$ and $\mathrm{Ker}(D)$ are topologically perfect.
\end{Proposition}

\begin{proof}
Let $\omega$ be a 
continuous 2-cocycle on $\fg \rtimes_{D} \R$. 
Define $\chi \colon \fg \rtimes_{D} \R \rightarrow \R$
by $\chi(D) := 0$ and
\[\chi(\xi) := \omega(D, I(\xi-\sigma \circ q \xi)) \quad \mbox{ for } \quad 
\xi \in \fg,\]
 where $q$, $\sigma$, $D$ and $I$ are as in the remarks following 
Definition~\ref{admissible}.
Then $\omega$ is cohomologous to $\omega' := \omega + \delta \chi$
which satisfies $\omega'(D , I D \fg ) = \{0\}$ because 
\[ (\delta\chi)(D,ID\xi) = -\chi([D, ID\xi]) = 
- \chi(DID\xi) = - \chi(D\xi) = - \omega(D,ID\xi).\] 
Therefore, the linear functional $i_{D} \omega' \colon \fg \rtimes_{D} \R \rightarrow \R$ 
factors through 
a map 
$\gamma_{\omega'} \colon \mathrm{Ker}(D) \rightarrow \R$. It vanishes on 
$[\mathrm{Ker}(D),\mathrm{Ker}(D)]$
by the cocycle property of $\omega'$, so we obtain a class
$[\gamma_{\omega'}] \in H^1(\mathrm{Ker}(D),\R)$.
Since $i_{D} \delta \chi|_{\mathrm{Ker}(D)} = 0$ for all 
$1$-chains $\chi$, the map $\omega' \mapsto \gamma_{\omega'}$
drops to a map
$H^{2}(\fg\rtimes_{D} \R,\R) \rightarrow H^1(\mathrm{Ker}(D),\R)$.

If $[\omega']\in \ker\gamma$, then $i_{D}\omega' = 0$ because 
$\g = \ker D \oplus ID\g$, so that 
$\omega'$ is determined by its restriction to $\fg$. There,  
the cocycle property 
\[\omega'([D,X],Y) + \omega'(X,[D,Y])) = \omega'(D,[X,Y])\]
implies $D$-invariance of $i_D\omega'$ on $\g$. 
Conversely, any $D$-invariant 2-cocycle $\omega$ on $\fg$
extends trivially to a cocycle $\tilde\omega$ on $\fg \rtimes_{D} \R$. 
This proves exactness in $H^2(\g \rtimes_D \R, \R)$. 

The relation $\beta([\omega]) = 0$ is equivalent to 
$\omega \in B^2(\fg,\R)^{D}$. Therefore, the kernel of $\beta$ 
is $B^2(\fg,\R)^{D} / B^{2}_{D}(\fg,\R)$, 
where the upper index $D$ stands for $D$-invariant coboundaries, 
and the lower index for coboundaries of $D$-invariant 1-chains.

The coboundary $\delta \phi \in B^2(\fg,\R)$ is determined by the restriction of $\phi \in C^1(\fg,\R) = \fg'$
to $\overline{[\fg,\fg]}$, and by the Hahn--Banach Theorem, 
the correspondence between $B^2(\fg,\R)$ and $\overline{[\fg,\fg]}{}'$ is bijective.
Now $\delta\phi \in B^2_D(\fg,\R)$ if and only if
$\phi \in \overline{[\fg,\fg]}{}'$ satisfies $\phi(D\fg \cap \overline{[\fg,\fg]}) = \{0\}$, and
$\delta\phi \in B^2(\fg,\R)^{D}$ if and only if 
$\phi(D \ol{[\fg,\fg]}) = \{0\}$. The quotient $B^2(\fg,\R)^{D} / B^{2}_{D}(\fg,\R)$ is therefore
equal to
\[ (\overline{[\fg,\fg]}/D[\fg,\fg])'/ (\overline{[\fg,\fg]}/(D\fg \cap \ol{[\fg,\fg]})' \simeq  
\left((D\fg \cap \ol{[\fg,\fg]}) / D\ol{[\fg,\fg]}\right)'.\qedhere\] 
\end{proof}

\begin{Remark}(Covariant cocycles for gauge algebras)\label{rk:perfect}
Suppose that $\fk$ is a perfect Lie algebra and that 
the equivariant $\R$-action is proper and
free on $M$ (Proposition~\ref{propfree}).  
Then, for $\fg = \Gamma_{c}(\Ad(P))$, we have
\[H^2(\fg\rtimes_{D}\R,\R) \simeq H^2_{D}(\fg,\R)\,.\]
Indeed, $\fg$ is topologically perfect because $\fk$ is perfect (\cite[Prop.~2.4]{JW13}).
Since 
$P_{\Sigma} \rightarrow \Sigma$ is a smooth principal fibre bundle 
(cf.\ Proposition~\ref{propfree}), the same
argument implies that
$\mathrm{Ker}(D) \simeq \Gamma_{c}(\Ad(P_{\Sigma}))$
is perfect, so that the statement follows from Proposition~\ref{DInvariant}.
%
\end{Remark}

\subsection{Cocycles for positive energy}\label{cocposen}

\begin{Proposition} {\rm (Spectral gap implies $D$-invariant cocycles)}
Let $D$ be an admissible derivation of $\fg$ and 
let $(\rho,\cH)$ be a smooth unitary 
representation of the Lie group $\widehat{G}$ with Lie algebra 
$\hat\g \cong \R \oplus_\omega (\g \rtimes_D \R)$. 
If $\mathrm{Spec}(-i\dd\rho(D)) \neq \R$, then
%
$\omega$ is co\-ho\-mo\-lo\-gous to an (automatically $D$-invariant) cocycle 
$\omega'$ with $\omega'(D ,\fg ) = \{0\}$.
\end{Proposition}
\begin{proof}
Recall that $\fg = ID\fg \oplus \mathrm{Ker}(D)$.
We have seen in the proof of Proposition~\ref{DInvariant} that $\omega$
is co\-ho\-mo\-lo\-gous to $\omega'$ with $\omega'(D, ID\fg) = \{0\}$.
Now let $X\in \mathrm{Ker}(D)$. 
Then $\dd\rho$ defines a unitary representation of the Heisenberg algebra
$(\R X \rtimes_{D} \R) \oplus_{\omega} \R c$. 
If $\omega'(D,X) \neq 0$, then the Stone-von Neumann Theorem
implies $\mathrm{Spec}(-i\dd\rho(D)) = \R$.
%
\end{proof}

We call a smooth unitary representation $\rho \colon \widehat{G} \rightarrow\U(\cH)$
\emph{periodic} if the corresponding $\R$-action factors through $\R \rightarrow \mathbb{T} 
\cong \R/\Z$,
and of \emph{positive energy} if the spectrum of the Hamilton operator
$H := -i\oline{\dd\rho(D)}$ 
is bounded below. Since $\mathrm{Spec}(H) \neq \R$ in either case, we have
the following corollary.

\begin{Corollary}{\rm (Periodic/positive energy implies $D$-invariant cocycles)}\label{Corrietje}
Suppose that $D$ is an admissible derivation of $\fg$ and that 
$\rho \colon \widehat{G} \rightarrow\U(\cH)$
is either a positive energy representation or a periodic representation.
Then the class $[\omega]$ is in the image
of the map $\beta \: H^2_{D}(\fg,\R)\rightarrow H^2(\fg\rtimes_{D}\R,\R)$ 
of {\rm Proposition~\ref{DInvariant}}.
\end{Corollary}

For periodic representations and positive energy representations of $\widehat{G}$ with an admissible
derivation, we may thus
assume w.l.o.g. that the cocycle $\omega$ is $D$-invariant and that $i_{D}\omega=0$.

\begin{Example} {\rm(Translation invariant cocycles for abelian gauge algebras)}
If, in the context of Definition~\ref{def:9.3}, 
 $K$ is abelian, then $\fg$ is the LF-space $\Gamma_{c}(\Ad(P)) = C^{\infty}_{c}(M,\fk)$
with the trivial Lie bracket, so cocycles on $\fg$ are skew symmetric 
$\fk$-valued distributions on $M \times M$.
By Corollary \ref{Corrietje}, every
cocycle on $\fg\rtimes_{D}\R$ that is derived from a periodic or 
positive energy representation, is cohomologous to one that
satisfies $i_{D}\omega = 0$, hence given (modulo coboundaries) by a 
$D$-invariant skew-symmetric $\fk$-valued distribution on $M\times M$.
%
\end{Example}

\section{Applications}\label{sec:10}

Finally, let us mention a few examples of locally convex Lie groups
that serve to illustrate the theory developed so far.
In Subsection~\ref{abeliangroups}, 
we consider
abelian Lie groups, whose central extensions are related to Heisenberg groups.
In Subsection~\ref{Virasorogroups}, we consider
the group $\Diff(\bS^1)_+$ of orientation preserving diffeomorphisms of the circle,
whose central extensions are related to the Virasoro algebra.
Finally, in Subsection~\ref{sec:loopgroups}, we investigate twisted loop groups,
whose central extensions are related to affine Kac--Moody 
algebras.

\subsection{Abelian Lie groups}\label{abeliangroups}

If $V$ is a locally convex vector space, then $(V,+)$ 
is an abelian locally convex Lie group.
The central extensions of the abelian Lie algebra $V$ correspond to
continuous skew-symmetric bilinear forms $\omega \colon V \times V \rightarrow \R$.
Each of these integrates to a group extension: 

\begin{Definition}(Heisenberg groups)
If $\omega$ is a continuous skew-symmetric bilinear form on a locally convex vector space $V$, 
then the group $\T \times_{\omega} V$ 
with product
\[
(z,v)(z',v') = (zz' \exp({\textstyle \frac{1}{2}} i \omega(v,v')), v + v')\,
\]
is a locally convex Lie group whose Lie algebra is 
$\R \oplus_{\omega} V$
with bracket
\[
(z,v)(z',v') = (\omega(v,v'), 0)\,.
\]
If $\omega$ is non-degenerate, then $\T \times_{\omega}V$ is called a \emph{Heisenberg group},
and $\R \oplus_{\omega}V$ a \emph{Heisenberg--Lie algebra}. 
\end{Definition}
Every central $\T$-extension of $(V,+)$ is isomorphic to $\T \times_{\omega}V$
for some continuous skew-symmetric bilinear form $\omega$ on $V$, 
and the central $\T$-extensions
$\T \times_{\omega}V$ and $\T \times_{\omega'}V$ are isomorphic 
if and only if $\omega = \omega'$ (\cite[Thm.~7.12]{Ne02}).
The smooth characters $\chi \colon V \rightarrow \T$ of the group $(V,+)$ are 
all of the form $\chi(v) = \exp(i\phi(v))$ for a continuous linear functional 
$\phi \colon V \rightarrow \R$.

Corollary~\ref{lineariseer} now
implies that, for every smooth projective unitary representation $(\ol\rho,\cH)$ of $(V,+)$,
there exists a uniquely determined continuous skew-symmetric bilinear form $\omega$ on $V$
such that $\ol\rho$ is induced by a
smooth unitary representation $(\widehat{\rho},\cH)$ of $\T\times_{\omega}V$.
By Proposition~\ref{autoriseer}, 
two representations $(\widehat{\rho}, \cH)$ and $(\widehat{\rho}', \cH')$ of $\T \times_{\omega} V$ 
give rise to equivalent projective representations of $(V,+)$
if and only if there exists a unitary transformation $U \colon \cH \rightarrow \cH'$ 
and a continuous functional $\phi \in V'$ such that $\widehat{\rho}'(z,v) = \exp(i\phi(v))U\widehat{\rho}(z,v)U^{-1}$.

This allows us to reduce the classification of smooth projective factor representations of $(V,+)$
to the the classification of smooth unitary factor representations of the Heisenberg group.
Indeed, let $(\widehat{\rho},\cH)$ be a smooth factor representation of $\T \times_{\omega} V$.
The closed subspace $K := {\rm Ker}(\omega) = \{v\in V\,;\, i_{v}\omega = 0\}$ 
is central in $\T\times_{\omega}V$, so that the requirement that $\widehat{\rho}$ be a 
factor representation, $\widehat{\rho}(V)' \cap \widehat{\rho}(V)'' = \C \one$, implies that 
$\widehat\rho|_{K}$ is given by a smooth character $\chi \colon K \rightarrow \T$.
Writing $\chi(k) = \exp(i\phi(k))$ and extending the continuous functional 
$\phi \colon K \rightarrow \R$ to $V$ (this is possible by the 
Hahn--Banach Theorem), we obtain a 
smooth character $\chi \colon V \rightarrow \T$ such that $\chi^{-1} \cdot \widehat{\rho}$
is trivial on $K$, hence factors through a representation of the locally convex Lie group
$\T \times_{\omega} (V/K)$. 
Since $K$ is closed, 
$V/K$ is a locally convex Hausdorff space, and if $V$ is barrelled, then so is $V/K$. 
However, unless $V$ is Fr\'echet, completeness of $V$
does \emph{not} imply completeness of $V/K$  (\cite[\S 31.6]{Ko69}), 
so at this point, we are making 
essential use of our rather wide definition of a locally convex Lie group, which does not assume completeness of the model space.


\begin{Proposition}{\rm (Projective $(V,+)$-rep's come from linear Heisenberg rep's)}
A smooth projective factor representation of $(V,+)$ is given by a continuous skew-symmetric 
bilinear form $\omega$ on $V$, together with a smooth unitary factor representation 
$(\widehat{\rho}, \cH)$ of the 
Heisenberg group $\T \times_{\omega} V/(\mathrm{Ker}(\omega))$ with $\widehat{\rho}(z) = z \one$.
Two such representations $(\widehat{\rho}, \cH)$ and $(\widehat{\rho}', \cH')$ 
give rise to equivalent projective representations of $(V,+)$
if and only if $\omega = \omega'$ and $\widehat{\rho}(z,v) = e^{i\phi(v)}U\widehat{\rho}(z,v)U^{-1}$
for a unitary transformation $U \colon \cH \rightarrow \cH'$ and a continuous linear 
functional $\phi \in (V/(\mathrm{Ker}\omega))'$.
\end{Proposition}

Switching from $V$ to $V/K$ and from $\rho$ to $\widehat{\rho}$ if necessary,
we may assume that
$(\rho,\cH)$ is a smooth representation of a Heisenberg group $\T\times_{\omega}V$.
By Proposition \ref{polarisability}, a smooth vector $\psi \in \cH^{\infty}$ defines a
non-degenerate continuous sesquilinear form
$H(\xi,\eta) = \langle \dd\rho(\xi) \psi, \dd\rho(\eta) \psi \rangle$ on $V_{\C}$ 
with $\omega = -2\mathrm{Im}(H)$.
In order to obtain the above equality, one has to require that
$\langle \psi , \dd\rho(\xi) \psi \rangle$ vanishes for $\xi \in \fg$, which can be achieved 
by twisting with $\phi(\xi) := i\langle \psi, \dd\rho(\xi) \psi \rangle$.

\begin{Proposition} {\rm(Characterisation of the relevant Heisenberg groups)}
A Heisenberg group $\T \times_{\omega} V$ possesses a smooth unitary representation $(\rho, \cH)$
with $\rho(z) = z \one$ if and only if there exists a Hermitian inner product $H$ on $V_{\C}$
such that $\omega = -2\mathrm{Im}(H)$.
\end{Proposition}

\begin{proof}
That existence of such an $H$ is necessary follows from the above. 
To see that it is sufficient, we note that the function 
\[ f \: \T \times_\omega V \to \C, \quad f(z,v) = z e^{- \shalf H(v,v)},\] 
is positive definite and leads to  the \emph{quasi-free} cyclic representation 
$(\rho, \cH,\Omega)$ (cf.~\cite[Thm.~3.4]{Pe90}) 
 with vacuum $\Omega \in \cH^{\infty}$ satisfying 
$\langle \Omega, \dd\rho_{H}(\xi) \Omega\rangle = 0$ and 
$\langle \dd\rho_{H}(\xi)\Omega, \dd\rho_{H}(\eta) \Omega\rangle = H(\xi,\eta)$
(see e.g.\ \cite[Thm.~3.8]{Pe90}, which draws heavily from \cite{MV68}).
\end{proof}

This completely reduces the theory of smooth projective unitary representations of $(V,+)$
to the theory of smooth unitary representations of the Heisenberg groups 
$\T \times_{\omega} V$ with $\omega = -2\mathrm{Im}H$ and $\rho(z) = z\one$, 
or equivalently, if $V$ is barrelled, to regular unitary representations $(\pi, W)$
of the Heisenberg--Lie algebra $\R \oplus_{\omega} V$ with $\pi(1,0) = 2\pi i\one$.
If $V$ is finite dimensional, then by the von Neumann Uniqueness Theorem \cite{vN31},
the Heisenberg group admits a single irreducible 
unitary representation with $\rho(z) = z\one$.
If $V$ is infinite dimensional, there exist many interesting 
irreducible and factor representations
besides the familiar Fock representations
\cite{MV68, DV71, D71, Ho82, Pe90}. Among these are the type III 
Araki--Woods factor representations,
which model equilibrium states at positive temperature 
in a free non-relativistic Bose gas \cite{AW63}.

\subsection{The Virasoro group}\label{Virasorogroups}

The group $\mathrm{Diff}(\bS^1)_+$ of orientation preserving diffeomorphisms 
of the circle is a connected locally convex Lie group with Lie algebra $\mathrm{Vec}(\bS^1)$
\cite{Ha82}. 
Its universal cover $G := \widetilde{\mathrm{Diff}}(S^1)_{+}$ can be described as
\[
G = \{\phi \in C^{\infty}(\R,\R): (\forall t \in\R)\  \phi(t + 2\pi) 
= \phi(t) + 2\pi \,\mathrm{and}\, \phi'(t) > 0\}\,,
\]
and $\pi^1(\mathrm{Diff}(\bS^1)) \simeq \Z$ is realised 
inside $G$ as the subgroup of functions
of the form $\phi_{n}(t) = t + 2\pi n$. 
The  universal central extension of $G$ is 
$\widehat{G} = \R \times_{B} \widetilde{\mathrm{Diff}}(S^1)_{+}$,
where $B$ is the globally defined Bott cocycle, 
\[
B(\phi, \psi) := \frac{1}{2}\int_{0}^{2\pi} \log((\phi \circ \psi)') d\log(\psi'),
\] 
and the induced product is $(\phi,a)(\psi,b) = (\phi\circ\psi, a + b + B(\phi,\psi))$ 
 \cite{Bo77}.
Note that $\pi_1(\mathrm{Diff}(S^1)_+) \simeq \Z$ is central in $\hat G$
because $B( \,\cdot\,, \phi_n) = B(\phi_n,\cdot) = 0$.

Its Lie algebra is the \emph{Virasoro algebra}
$\mathrm{Vir} := \R \oplus_{\omega} \mathrm{Vec}(S^1)$, where $\omega$ is the 
\emph{Gel'fand--Fuks cocycle}
\[\omega_{1}(\xi\partial_{t},\eta\partial_{t}) = -\frac{i}{2} \int_{0}^{2\pi} 
(\xi'\eta'' - \eta'\xi'')dt \] 
(see \cite[Ch.~II, \S~2]{KW09}). 
Since $G$ is contractible, every Lie algebra cocycle $\omega$ integrates to 
the group level (\cite[Thm.~7.12]{Ne02}). 
If $\omega$ is cohomologous to $c\omega_{1}$ with $c\neq 0$, then 
the corresponding group extension $\T \rightarrow G^{\sharp}\rightarrow G$ is isomorphic to 
the quotient of $\widehat{G}$ by the central subgroup $c\Z \subseteq \R$.

From Theorem \ref{maintheorem}.A), we then obtain that smooth projective unitary 
representations $(\overline{\rho}, \cH)$ of $G$ correspond to smooth 
unitary representations 
$(\widehat{\rho}, \cH)$ of $\widehat{G}$
with $\widehat{\rho}((t,0))= e^{2\pi i c t}\one$ for some $c\in \R$.
It comes from a projective $\mathrm{Diff}(S^1)_+$-representation if and only if 
$\widehat{\rho}(\phi_1) = \gamma \one\in \T\1$. 
Since $\R \times\pi_1(\mathrm{Diff}(S^1))$ is central in $\widehat{G}$, such pairs 
$(c,\gamma)$ automatically exist if $\widehat{\rho}$ is an irreducible or, more generally, 
factor representation of $\widehat{G}$.
By Theorem~\ref{maintheorem}.C, one can alternatively characterise the smooth 
projective unitary representations of $G$ as the regular representations 
$(\pi,V)$ of $\mathrm{Vir}$ with $\pi(1,0) = 2\pi i c \one$.
See \cite{GKO86, KR87, L88} for the Lie algebra and
\cite{GW85, TL99, Se81, NS14} for the group representations.


\subsection{Loop groups}\label{sec:loopgroups}

Let $K$ be a compact simple Lie group with Lie algebra $\fk$,
and let $\sigma$ be an automorphism of $K$ of finite order $N$. 
We denote the corresponding Lie algebra automorphism by the same letter,
and assume that it comes from a diagram automorphism of $\fk$ (\cite{Ka90}).
Then the  \emph{(twisted) loop group}
\[
\cL_{\sigma}(K) := \{f \in C^{\infty}(\R,K) : (\forall t \in \R)\ f(t+1) = \sigma^{-1}f(t)\}
\]
is a Fr\'echet--Lie group
\cite[App.~A]{NW09}.
Its Lie algebra is the {\it(twisted) loop algebra} 
\[
\cL_{\sigma}(\fk) := \{\xi \in C^{\infty}(\R,\fk): (\forall t \in \R)
\ \xi(t+1) = \sigma^{-1}\xi(t)\}\,.\]

It comes with a canonical 1-parameter group 
$T \colon \R \rightarrow \mathrm{Aut}(\cL_{\sigma}(K))$ of translations defined by 
$T_{\tau}(f) (t) := f(t + \tau)$.
The corresponding derivation $D = \partial_{t}$ of $\cL_{\sigma}(\fk)$
is admissible by Proposition~\ref{prop:circleaction} because $T$ is periodic.


We are interested in smooth covariant projective unitary representations 
$(\ol\rho, \ol{U}, \cH)$ of the pair $(\cL_{\sigma}(K), \R)$.
The first step is to determine the relevant group extensions.
Note that ${H^2(\cL_{\sigma}(\fk) \rtimes_{\partial_{t}} \R, \R)}$
is isomorphic to ${H^2(\cL_{\sigma}(\fk), \R)_{D}}$.
This follows from Proposition~\ref{DInvariant}
because both $\cL_{\sigma}(\fk)$  and $\mathrm{Ker}(D)$ are topologically perfect;
the former by \cite[Prop.~2.4]{JW13}, and the latter because
$\mathrm{Ker}(D)$ is the fixed point Lie algebra $\fk^{\sigma}$, which 
is compact semisimple, hence perfect, by \cite[Ch.~X, Thm.~5.15]{He78}.
Since ${H^2(\cL_{\sigma}(\fk), \R)}$ is 1-dimensional, and 
its generator $[\omega_1]$
possesses a translation invariant representative 
(cf.\ \cite[Ch.\ 7, 8]{Ka90})
\[\omega_{1}(\xi,\eta) = \frac{1}{8\pi}\int_{0}^{N} \kappa(\xi,{\textstyle \frac{d}{dt}}\eta) dt,\]
where $\kappa$ is a suitably normalized invariant symmetric bilinear 
form on $\fk$, it is isomorphic to ${H^2(\cL_{\sigma}(\fk), \R)_{D}}$.
The Lie algebra extension 
${\widetilde{\cL}_{\sigma}(\fk) \rightarrow \cL_{\sigma}(\fk)}$ 
corresponding 
to the class $c[\omega_{1}]$ for $c\neq 0$
is called the extension at \emph{level} $c$.
(For different values of $c\neq 0$, the Lie algebras $\widetilde{\cL}_{\sigma}(\fk)$ are
isomorphic, but the central extensions are not.)
It integrates to a group extension $\widehat{G}_{c} \rightarrow \cL_{\sigma}(K)_0$
if and only if\label{cinz} $c \in \Z$. 
These extensions have their origins in the Wess--Zumino--Novikov--Witten 
model \cite{Wi84, Pi89}.
Since the time translation 
is periodic, it has an eigenvector, so Remark~\ref{rem:inv-cocyc} applies and 
the 1-parameter group of automorphisms extends to $\widehat{G}_{c}$.
In combination with
Theorem~\ref{maintheorem}.A, we now find that equivariant smooth projective unitary representations 
$(\ol{\rho}, \ol{U}, \cH)$ of $(\mathcal{L}_{\sigma}(K)_{0}, \R)$ correspond to 
equivariant smooth unitary representations
$(\widehat{\rho}, U, \cH)$
of $(\widehat{G}_{c}, \R)$ with $\widehat{\rho}(z) = z\one$ for $z\in \T$.

Unfortunately, unlike in the Virasoro case, the extension 
$\widehat{G}_{c}\rightarrow \mathcal{L}_{\sigma}(K)_{0}$
is nontrivial as a principal $\T$-bundle, which makes its
construction less explicit and rather more complicated \cite{KW09}.
We therefore turn to 
part C of Theorem~\ref{maintheorem}, 
which translates the projective unitary representations
of the group $\widehat{G}_{c} \rtimes_{\partial_{t}} \R$
to regular unitary
representations of the more tractable Lie algebra 
\[\widehat{\cL}_{\sigma}(\fk) := \widetilde{\cL}_{\sigma}(\fk) \rtimes_{\partial_{t}}\R\,, \]
the \emph{completed affine Kac--Moody algebra}.

In order to apply Theorem~\ref{maintheorem}.C, we need 1-connected Lie groups.
Although we can always
go to the universal cover of $\cL_{\sigma}(K)_{0}$ and argue as in 
Subsection~\ref{Virasorogroups}, the following proposition shows that 
this is not necessary in the important special case when $K$ is 1-connected.
\begin{Proposition} {\rm(Topology of twisted loop groups)}
If $K$ is simply connected (1-connected), then so are the groups $\cL_{\sigma}(K)$
and 
$\cL_{\sigma}(K) \rtimes_{\partial_{t}} \R$. 
\end{Proposition}
\begin{proof}
Since $\R$ is contractible, it suffices to consider $\cL_{\sigma}(K)$.
The evaluation in zero, $\mathrm{ev_0} \colon \cL_{\sigma}(K) \rightarrow K$, is a smooth Lie group extension
whose image $K^{[\sigma]}$ is an open
subgroup  of $K$, and whose kernel satisfies $\pi_{1}(\mathrm{Ker}(\mathrm{ev}_0)) \simeq \pi_2(K) = \{1\}$
\cite[\S 3]{NW09}. 
It follows that if $K$ is simply connected, then
$\pi_1(K^{[\sigma]})$ and $\pi_2(K)$ are trivial, so that the long homotopy sequence
\[
\ldots \rightarrow 
\pi_1(\mathrm{Ker}(\mathrm{ev}_0)) \rightarrow \pi_1(\cL_{\sigma}(K)) \rightarrow \pi_1(K^{[\sigma]})
\rightarrow \ldots
\] 
implies $\pi_1(\cL_{\sigma}(K)) = \{0\}$. If $K$ is 1-connected, then a similar line of reasoning 
implies that $\cL_{\sigma}(K)$ is 1-connected
(cf.~\cite[Rk.~3.6.A]{NW09}).
\end{proof}
It follows that if $K$ is 1-connected, then
smooth equivariant  projective unitary representations 
$(\ol{\rho}, \ol{U}, \cH)$ of $(\mathcal{L}_{\sigma}(K), \R)$ correspond to 
regular unitary representations
$(\pi, V)$ of the completed affine Kac--Moody algebra
$\widehat{\cL}_{\sigma}(\fk)$
at integral level $c$, that is, with $\pi(1) = c \one$ for some $c\in \Z$.

See \cite{Se81, PS86} for
the positive energy representations of $\widehat{G}_c$, and
\cite{Ka80, Ka90} for the (corresponding) highest weight representations of 
$\widehat{\g}_{c}$ (in the algebraic context).

\bibliographystyle{alpha}

\end{document}